\documentclass{article}
\usepackage{amsmath,amsfonts,amssymb,amsthm,cases,graphicx}
\usepackage{enumitem}
\usepackage{tikz}
\usepackage{hyperref}
\usepackage{color}
\usepackage{vmargin}
\setmarginsrb{2cm}{1cm}{2cm}{3cm}{1cm}{1cm}{2cm}{2cm}
\usepackage{authblk}
\usepackage[capitalise]{cleveref}

\newcommand{\ov}{\overset{ \mbox{\tiny def}}{=}}
\renewcommand{\div}{\operatorname{div}}

\newcommand{\Cof}{\operatorname{Cof}}

\renewcommand{\leq}{\leqslant}

\renewcommand{\geq}{\geqslant}

\renewcommand{\Re}{\operatorname{Re}}

\newcounter{tak}

\numberwithin{equation}{section}
\everymath{\displaystyle}
\reversemarginpar
\newtheorem{Theorem}{Theorem}[section]

\newtheorem{Corollary}[Theorem]{Corollary}
\newtheorem{Proposition}[Theorem]{Proposition}
\newtheorem{Lemma}[Theorem]{Lemma}
\newtheorem{Remark}[Theorem]{Remark}

\title{Gevrey regularity for a system coupling the Navier-Stokes system with a beam: the non-flat case} 

\author[1]{Mehdi Badra}
\author[2]{Tak\'eo Takahashi}
\affil[1]{Institut de Math\'ematiques de Toulouse ; UMR5219;
Universit\'e de Toulouse ; CNRS ; UPS IMT, F-31062 Toulouse Cedex 9, France\\
\it{mehdi.badra@math.univ-toulouse.fr }}
\affil[2]{Universit\'e de Lorraine, CNRS, Inria, IECL, F-54000 Nancy, France, \it{takeo.takahashi@inria.fr}}

\date{\today}

\begin{document}
\maketitle

\abstract{
We consider a bi-dimensional viscous incompressible fluid in interaction with 
a beam located at its boundary. We show the existence of strong solutions for this fluid-structure interaction system, extending a previous result
\cite{plat} where we supposed that the initial deformation of the beam was small. 
The main point of the proof consists in the study of the linearized system and in particular in proving that the corresponding semigroup is of Gevrey class. 
}

\vspace{1cm}



\tableofcontents

\newpage
\section{Introduction}
This work is devoted to the mathematical analysis of a fluid-structure interaction system where the fluid is modeled by the Navier-Stokes system whereas the structure is a beam situated at a part of the fluid domain. We consider here the bi-dimensional case in space, that is the fluid domain is a subset of $\mathbb{R}^2$ whereas the beam domain is an interval. Another important assumption for our analysis is to assume periodic boundary conditions in the direction orthogonal to the beam deformation. To be more precise, let $L>0$ be the length of the beam and let us set
\begin{equation}\label{np004}
\mathcal{I}\ov \mathbb{R}/L\mathbb{Z}.
\end{equation}
For any deformation $\eta : \mathcal{I} \to (-1,\infty)$, we also consider the corresponding fluid domain
\begin{equation}\label{gev0.0}
\mathcal{F}_{\eta}\ov  \left \{(x_1,x_2)\in \mathcal{I}\times \mathbb{R} \ ; \ x_2\in (0,1+\eta(x_1))\right \}.
\end{equation}
The boundary of $\mathcal{F}_{\eta}$ can be splitted into a ``deformable'' part
\begin{equation}\nonumber
\Gamma_{\eta} \ov  \left \{(s,1+\eta(s)), \; s\in \mathcal{I} \right \},
\end{equation}
and a ``fixed'' part
\begin{equation}\nonumber
\Gamma_{-1} \ov  \mathcal{I}\times \{0\}.
\end{equation}
We recall the geometry in Figure \ref{F1}.
\begin{figure}\label{F1}
\begin{center}
\begin{tikzpicture}
\draw (0,0) -- (4,0);
\draw (0,2) -- (0,0);
\draw (4,0) -- (4,2);
\draw [domain=0:4,samples=30] plot (\x, {2+0.1*sin(pi/2*\x r)-0.2*sin(pi*\x r)});
\draw (2,1) node {$\mathcal{F}_{\eta}$};
\draw (2,0) node[below] {$\Gamma_{-1}$};
\draw (0,1) node {$-$};
\draw (0,0.95) node {$-$};
\draw (4,1) node {$-$};
\draw (4,0.95) node {$-$};
\draw (2,2) node[above] {$\Gamma_{\eta}$};
\draw (0,0) node[below] {$0$};
\draw (4,0) node[below] {$L$};
\draw[dashed] (4,0) -- (8,0);
\draw[dashed] (8,0) -- (8,2);
\draw [dashed, domain=4:8,samples=30] plot (\x, {2+0.1*sin(pi/2*\x r)-0.2*sin(pi*\x r)});
\draw[dashed] (-4,0) -- (0,0);
\draw[dashed] (-4,2) -- (-4,0);
\draw [dashed, domain=-4:0,samples=30] plot (\x, {2+0.1*sin(pi/2*\x r)-0.2*sin(pi*\x r)});
\end{tikzpicture}
\end{center}
\caption{Our geometry}
\end{figure}

Let us denote by $v$ and $p$ the velocity and the pressure of the fluid. Then, the system modeling the interaction between the viscous incompressible fluid and the beam is
\begin{equation}\label{tak2.3}
\left\{\begin{array}{c}
\partial_t v +(v\cdot \nabla) v-\div \mathbb{T}(v,p) = 0, \quad t>0, \  x\in \mathcal{F}_{\eta(t)},\\
\div v =   0,   \quad t>0, \ x\in \mathcal{F}_{\eta(t)},\\
v(t,s,1+\eta(t,s)) = (\partial_t \eta)(t,s) e_2, \quad t>0, \ s\in \mathcal{I},\\
v = 0  \quad t>0, \ x\in \Gamma_{-1},\\
\partial_{tt} \eta + \alpha_1 \partial_{ssss} \eta-\alpha_2 \partial_{ss} \eta 
  =-\widetilde{\mathbb{H}}_{\eta}(v,p), \quad t>0, \  s\in \mathcal{I},\\
\end{array}\right.
\end{equation}
with the initial conditions
\begin{equation}\label{su4.3}
\eta(0,\cdot)=\eta_1^0, \quad \partial_{t} \eta(0,\cdot)=\eta_2^0 \quad\mbox{ and }\quad  v(0, \cdot)=v^0 \ \text{in} \ \mathcal{F}_{\eta_1^0}.
\end{equation}

The two first equations correspond to the Navier-Stokes system, whereas the last equation is the beam equation. We have considered the no-slip boundary conditions (third and forth equations). The canonical basis of $\mathbb{R}^2$ is denoted by $(e_1, e_2)$ and we have also used the following notations:
\begin{equation}\label{ws4.6}
\mathbb{T}(v,p) \ov 2 \nu D(v)- p I_2,\quad D(v)= \frac{1}{2}\left(\nabla v + (\nabla v)^*\right),
\end{equation}
\begin{equation}\label{ws4.7}
\widetilde{\mathbb{H}}_{\eta}(v,p)\ov \left\{ (1+|\partial_s \eta|^2)^{1/2} \left[\mathbb{T}(v,p)n\right](t,s,1+\eta(t,s))\cdot e_2 \right\}.
\end{equation}
We assume that the constants satisfy
$$
\nu>0 \ \text{(viscosity)}, \quad \alpha_1>0, \quad \alpha_2\geq 0.
$$
Finally, the vector fields $n$ is the unit exterior normal to $\mathcal{F}_{\eta(t)}$: $n=-e_2$ on $\Gamma_{-1}$ and on $\Gamma_{\eta(t)}$,
\begin{equation}\label{DefnormalStruct}
n(t,x_1,x_2)=\frac{1}{\sqrt{1+|\partial_s \eta(t,x_1)|^2}}
\begin{bmatrix}
-\partial_s \eta(t,x_1)
\\
1
\end{bmatrix}. 
\end{equation}

An important remark is that a solution to \eqref{tak2.3} satisfies
$$
\frac{d}{dt} \int_0^L \eta(t,s)\ ds=0. 
$$
By assuming that the mean value of $\eta_1^0$ is zero, this leads to
\begin{equation}\label{ws7.1}
\int_0^L \eta(t,s)\ ds = 0 \quad (t\geq 0).
\end{equation}
We denote by $M$ the orthogonal projection from $L^2(\mathcal{I})$ onto $L^2_{0}(\mathcal{I})$
where
\begin{equation}\label{gev9.0}
L^2_{0}(\mathcal{I}) \ov \left\{f \in L^2(\mathcal{I}) \ ; \ \int_0^L f(s) \ ds=0  \right\}.
\end{equation}
Taking the projection of the last equation of \eqref{tak2.3} on $L^2_{0}(\mathcal{I})$ gives
\begin{equation}\label{beamabs}
\partial_{tt} \eta + A_1\eta 
  =-{\mathbb{H}}_{\eta}(v,p), \quad t>0, \  s\in \mathcal{I},
\end{equation}
where 
\begin{equation}\label{np0060}
{\mathbb{H}}_{\eta}(v,p)\ov M\widetilde{\mathbb{H}}_{\eta}(v,p),
\end{equation}
and where $A_1$ is the operator for the structure defined by
\begin{equation}\label{gev3.3}
\mathcal{H}_{\mathcal{S}} \ov L^2_{0}(\mathcal{I}), \quad \mathcal{D}(A_1)\ov H^4(\mathcal{I})\cap L^2_{0}(\mathcal{I}),
\end{equation}
\begin{equation}\label{gev3.3bis}
A_1 : \mathcal{D}(A_1)\to \mathcal{H}_{\mathcal{S}}, \quad \eta \mapsto \alpha_1 \partial_{ssss} \eta -\alpha_2 \partial_{ss}\eta.
\end{equation}
One can check that for any $\theta\geq 0$,
\begin{equation}\label{gev3.4}
\mathcal{D}(A_1^\theta) =  H^{4\theta}(\mathcal{I})\cap L^2_{0}(\mathcal{I}).
\end{equation}

The projection of the last equation of \eqref{tak2.3} on $L^2_{0}(\mathcal{I})^\perp$ allows us to determine the constant for the pressure (see \cite{plat} for more details): at the contrary to the classical Navier-Stokes system without structure, here the pressure is not determined up to a constant. 

The classical Lebesgue and Sobolev spaces  are denoted by $L^\alpha$, $H^k$ and we use the notation $C^0$ for the space of continuous maps and $C^0_b$ for the space of continuous and bounded maps. 
We use the bold notation for the spaces of vector fields: ${\bf L}^\alpha=(L^\alpha)^2$, ${\bf H}^k=(H^k)^2$ etc. 
Since the fluid domain is moving, we introduce spaces of the form 
$H^1(0,T;L^q(\mathcal{F}_{\eta}))$, $L^2(0,T;H^k(\mathcal{F}_{\eta}))$, etc. with $T\leq \infty$. 
If $\eta(t,\cdot)>-1$ $(t\in (0,T))$, then 
$$
v\in H^1(0,T;L^q(\mathcal{F}_{\eta})) \quad \text{if} \quad
y\mapsto v(t,y_1, y_2(1+\eta(t,y_1))\in H^1(0,T;L^q(\mathcal{F}_{0}))
$$ 
and similarly, for the other spaces. We also write
$$
H^\alpha_{0}(\mathcal{I}) \ov H^\alpha(\mathcal{I})\cap L^2_{0}(\mathcal{I}) \quad (\alpha \geq 0).
$$
Finally, we use $C$ as a generic positive constant that does not depend on the other terms of the inequality.
The value of the constant $C$ may change from one appearance to another.

Let us write our hypotheses for the initial conditions: there exists $\varepsilon>0$ such that
\begin{equation}\label{bea8.4}
\eta^0_1\in W^{7,\infty}(\mathcal{I})\cap L_{0}^{2}(\mathcal{I}), 
\quad
\eta^0_2\in H_{0}^{1+\varepsilon}(\mathcal{I})
\quad 
\eta_1^0 > -1 \quad \text{in} \ \mathcal{I},
\end{equation}
\begin{equation}\label{bea8.5}
v^0\in {\bf H}^1(\mathcal{F}_{\eta_1^0}),
\end{equation}
with
\begin{gather}
\div v^0=0 \ \text{in} \  \mathcal{F}_{\eta_1^0},\quad
v^0(s,1+\eta_1^0(s)) = \eta_2^0(s) e_2   \quad s\in \mathcal{I},\quad
v^0 = 0  \quad \text{on} \ \Gamma_{-1}.\label{bea8.7}
\end{gather}

Our main result on \eqref{tak2.3} is the existence and uniqueness of strong solutions for small times:
\begin{Theorem}\label{T01}
For any $[v^0, \ \eta^0_1, \ \eta^0_2]$ satisfying \eqref{bea8.4}--\eqref{bea8.7}, 
there exist $T>0$ and a strong solution 
$(\eta, v, p)$ of \eqref{tak2.3} with
\begin{equation}\label{contact1}
\eta(t,\cdot) > -1 \quad t\in [0,T],
\end{equation}
\begin{equation}\label{reg-vp}
v\in L^2(0,T;{\bf H}^2(\mathcal{F}_{\eta})\cap C^0([0,T];{\bf H}^1(\mathcal{F}_{\eta})) \cap H^1(0,T;{\bf L}^2(\mathcal{F}_{\eta})),
\quad
p\in L^2(0,T;H^1(\mathcal{F}_{\eta})),
\end{equation}
\begin{equation}\label{reg-eta}
\begin{array}{c}
\eta \in L^2(0,T;H_{0}^{7/2}(\mathcal{I}))\cap C^0([0,T];H_{0}^{5/2}(\mathcal{I})) \cap H^1(0,T;H_{0}^{3/2}(\mathcal{I})),
\\[0.1cm]
\partial_t \eta \in L^2(0,T;H_{0}^{3/2}(\mathcal{I}))\cap C^0([0,T];H_{0}^{1/2}(\mathcal{I})) \cap H^1(0,T;(H_{0}^{1/2}(\mathcal{I}))'),
\end{array}
\end{equation}
the first four equations of \eqref{tak2.3} are satisfied almost everywhere or in the trace sense 
and \eqref{beamabs} holds in $L^2(0,T;H_{0}^{1/2}(\mathcal{I})')$.

{This solution is unique locally: if $(\eta^{(*)}, v^{(*)}, p^{(*)})$ is another solution with the same regularity, there exists $T^*>0$ such that
$$
(\eta^{(*)}, v^{(*)}, p^{(*)})=(\eta, v, p) \quad \text{on} \ [0,T^*].
$$}
\end{Theorem}

In order to prove the above result, a first standard step consists in rewriting the Navier-Stokes system in the fixed spatial domain 
\begin{equation}\label{F0}
\mathcal{F}\ov \mathcal{F}_{\eta_1^0},
\end{equation}
by using a change of variables.
Then, one of the main ingredients to obtain \cref{T01} is a result on a linear system associated with \eqref{tak2.3}:
\begin{equation}\label{gev0.3}
\left\{\begin{array}{c}
\partial_t w -\div \mathbb{T}(w,q) = F, \quad t>0, \  y\in \mathcal{F},\\
\div w =   0   \quad t>0, \ y\in \mathcal{F},\\
w(t,s,1+\eta_1^0(t,s)) = (\partial_t \eta)(t,s) e_2  \quad t>0, \ s\in \mathcal{I},\\
\partial_{tt} \eta + A_1 \eta 
  =-\mathbb{H}_{\eta^0_1}(w,q)+G, \quad t>0, \  s\in \mathcal{I},\\
\end{array}\right.
\end{equation}
with the initial conditions 
\begin{equation}\label{bis}
w(0,\cdot)=w^0,
\quad 
\eta(0,\cdot)= \zeta_1^0,
\quad 
\partial_t \eta(0,\cdot)= \zeta_2^0.
\end{equation}

For this system, we have the following result
\begin{Theorem}\label{T04}
Assume
$$
\eta_1^0 \in W^{7,\infty}(\mathcal{I}),\quad \eta_1^0>-1 \quad \text{in} \ \mathcal{I}.
$$
Suppose $F\in L^2(0,\infty;{\bf L}^2(\mathcal{F}))$ and $G\in L^2(0,\infty;\mathcal{D}(A_1^{1/8}))$, $\varepsilon>0$,
\begin{equation}\label{bea8.4L}
\zeta^0_1\in H_{0}^{3+\varepsilon}(\mathcal{I}), 
\quad
\zeta^0_2\in H_{0}^{1+\varepsilon}(\mathcal{I}),
\quad
w^0\in {\bf H}^1(\mathcal{F}),
\end{equation}
\begin{equation}
\div w^0=0 \ \text{in} \  \mathcal{F},
\quad
w^0(s,1+\zeta_1^0(s)) = \zeta_2^0(s) e_2   \quad s\in \mathcal{I},
\quad
w^0 = 0  \quad \text{on} \ \Gamma_{-1}.\label{bea8.7L}
\end{equation}
Then \eqref{gev0.3}-\eqref{bis} admits a unique solution
\begin{equation}\label{bea0.8}
w\in L^2(0,\infty;{\bf H}^2(\mathcal{F}))\cap C^0_b([0,\infty);{\bf H}^1(\mathcal{F})) \cap H^1(0,\infty;{\bf L}^2(\mathcal{F})),
\quad
q\in L^2(0,\infty;H^1(\mathcal{F})/\mathbb{R} ),
\end{equation}
\begin{equation}\label{bea0.9}
\eta \in L^2(0,\infty;\mathcal{D}(A_1^{7/8}))\cap C^0_b([0,\infty);\mathcal{D}(A_1^{5/8})) \cap H^1(0,\infty;\mathcal{D}(A_1^{3/8})),
\end{equation}
and
\begin{equation}\label{bea1.0}
\partial_t \eta \in L^2(0,\infty;\mathcal{D}(A_1^{3/8}))\cap C^0_b([0,\infty);\mathcal{D}(A_1^{1/8})) \cap H^1(0,\infty;\mathcal{D}(A_1^{1/8})').
\end{equation}
Moreover, there exists $C_{0}>0$ such that
\begin{multline}\label{np0100}
\left\| w \right\|_{L^2(0,\infty;{\bf H}^2(\mathcal{F}))\cap C^0_b([0,\infty);{\bf H}^1(\mathcal{F})) \cap H^1(0,\infty;{\bf L}^2(\mathcal{F}))}
+\|q\|_{L^2(0,\infty;H^1(\mathcal{F})/\mathbb{R})}
\\
+\|\eta\|_{L^2(0,\infty;\mathcal{D}(A_1^{7/8}))\cap C^0_b([0,\infty);\mathcal{D}(A_1^{5/8})) \cap H^1(0,\infty;\mathcal{D}(A_1^{3/8}))}
\\
+\|\partial_t \eta \|_{L^2(0,\infty;\mathcal{D}(A_1^{3/8}))\cap C^0_b([0,\infty);\mathcal{D}(A_1^{1/8})) \cap H^1(0,\infty;\mathcal{D}(A_1^{1/8})')}
\\
\leq
C_{0}\left(
\|w^0\|_{{\bf H}^1(\mathcal{F})}
+ 
\|\zeta^0_1\|_{\mathcal{D}(A_1^{3/4+\varepsilon})}
+
\|\zeta^0_2\|_{\mathcal{D}(A_1^{1/4+\varepsilon})}
\right.
\\
+
\left.
\|F\|_{L^2(0,\infty;{\bf L}^2(\mathcal{F}))} 
+
\|G\|_{L^2(0,\infty;\mathcal{D}(A_1^{1/8}))}
\right).
\end{multline}
\end{Theorem}
In \cite{plat}, we obtained \cref{T04} only in the case $\eta_1^0= 0$ so that the result on \eqref{tak2.3} was reduced to the case of small initial deformations.
Here we are no longer restricted to this hypothesis. As in \cite{plat}, the proof of \cref{T04} relies on resolvent estimates and results on semigroup of Gevrey class. More precisely, it is a consequence of \cref{CorMain}.

\begin{Remark}
As explained above, the main novelty here is to remove the restriction of smallness of $\eta_1^0$ that was needed in \cite{plat}. Our method to obtain the result for the linear system is based on commutator estimates (see \cref{sec_com}). The main drawback of such approach is that we need a more regular initial deformation ($W^{7,\infty}$ instead of $H^{3+\varepsilon}$). Even without this condition, we have as in our previous result a loss of regularity for 
$(\eta, \partial_t \eta)$: the continuity of $(\eta, \partial_t \eta)$ lies in $H^{5/2}(0,L)\times H^{1/2}(0,L)$ 
but we need to impose that at initial time, it belongs to $W^{7,\infty}(0,L)\times H^{1+\varepsilon}(0,L)$ for some $\varepsilon>0$. 
This is due to this model that couples two dynamical systems of different nature and in particular the linear system \eqref{gev0.3} couples
the Stokes system and the beam equation and the corresponding semigroup is not analytic but only of Gevrey class as stated in \cref{T04}. 

With an appropriate damping on the beam equation, we can recover an analytic semigroup. More precisely, in the original model 
proposed in \cite{QuaTuvVen2000a} (for the blood flow in a vessel), the beam equation in \eqref{tak2.3}  is replaced by
\begin{equation}\label{EqStructintro}
\partial_{tt} \eta + \alpha_1 \partial_{ssss} \eta-\alpha_2 \partial_{ss} \eta-\delta \partial_{tss} \eta
	=-\widetilde{\mathbb{H}}_{\eta}(v,p),
\end{equation}
with $\delta>0$. 

Several works analyze such a model: \cite{ChaDesEst2005a} (existence of weak solutions), \cite{Bei2004a}, \cite{Leq2011a} and \cite{MR3466847} (existence of strong solutions),
\cite{RaymondSICON2010} (stabilization of strong solutions), \cite{MR3619065} (stabilization of weak solutions around a stationary state). In all these works, 
the damping term $-\delta \partial_{tss} \eta$ is crucial. Few works have tackled  the case without damping: the existence of weak solutions is proved in \cite{Gra2008a}. In \cite{grandmont:hal-01567661}, 
the existence of local strong solutions is obtained 
for a structure described by either a wave equation ($\alpha_1=\delta=0$ and $\alpha_2>0$ in \eqref{EqStructintro}) 
or a beam equation with inertia of rotation ($\alpha_1>0$, $\alpha_2=\delta=0$ and with an additional term $-\partial_{ttss}\eta$ in \eqref{EqStructintro}). 
Finally, in our previous work \cite{plat} we proved the existence and uniqueness of strong solutions in the case  of an undamped beam equation but for small initial deformations.
\end{Remark}

The outline of the article is as follows: in \cref{sec_classical}, we construct and use a change of variables to write system \eqref{tak2.3} in a cylindrical domain and then linearize it. 
\cref{sec_op} is devoted to the introduction of several useful operators together with their properties. In order to prove \cref{T04} we need to estimate commutators appearing due the fact that our initial domain $\mathcal{F}$ is not flat. 
Such estimates allows us to deduce resolvent estimates in \cref{sec_res} by estimating the inverse of the operator ${V}_{\lambda}$  (see \eqref{EqresolventA0}). At first, we first estimate an approximation of ${V}_{\lambda}^{-1}$ 
in \cref{sec_approx}. Finally, in \cref{sec_fix} we recall the idea of the proof of \cref{T01} based on \cref{T04}, by using a fixed point argument.

\section{Change of variables and linearization}\label{sec_classical}
\subsection{The system written in a fixed domain}\label{sec_construc}
In this section, we defined and use a standard change of variables to rewrite system \eqref{tak2.3} in a cylindrical domain.
We set
\begin{equation}\label{gev9.7}
X_{\eta^1,\eta^2} : \mathcal{F}_{\eta^1} \to \mathcal{F}_{\eta^2}, 
\quad 
(y_1,y_2)\mapsto \left(y_1, y_2\frac{1+\eta^2(y_1)}{1+\eta^1(y_1)}\right),
\end{equation}
whose inverse is $X_{\eta^2,\eta^1}$.
In our case, we consider
\begin{gather}
X(t,\cdot)\ov X_{\eta_1^0,\eta(t)} : (y_1,y_2)\mapsto \left(y_1, y_2\frac{1+\eta(t,y_1)}{1+\eta_1^0(y_1)}\right), 
\label{bea3.4}
\\
Y(t,\cdot)\ov X(t,\cdot)^{-1}=X_{\eta(t),\eta_1^0} : (x_1,x_2)\mapsto \left(x_1, x_2\frac{{1+\eta_1^0(x_1)}}{1+\eta(t,x_1)}\right),
\label{bea3.5}
\end{gather}
so that $X(t,\cdot)$ transforms $\mathcal{F}=\mathcal{F}_{\eta^0}$ onto $\mathcal{F}_{\eta(t)}$.
Then, we write
\begin{equation}\label{gev8.2}
a\ov \Cof(\nabla Y)^*, \quad b\ov \Cof(\nabla X)^*,
\end{equation}
\begin{equation}\label{gev8.0bis}
w(t,y)\ov b(t,y) v(t,X (t,y))\quad \mbox{ and }\quad
q(t,y)\ov p(t,X (t,y)),
\end{equation}
so that
\begin{equation}\label{gev8.1bis}
v(t,x)=a(t,x) w(t,Y(t,x)) \quad \text{and} \quad 
p(t,x)=q(t,Y(t,x)).
\end{equation}
After some calculation (see for instance \cite{plat}), system \eqref{tak2.3}, \eqref{su4.3} rewrites,
\begin{equation}\label{tak2.3-CV2}
\left\{\begin{array}{c}
\partial_t w -\div \mathbb{T}(w,q) = \widehat{F}(\xi,w,q) \quad \text{in} \ (0,\infty)\times \mathcal{F},\\
\div w =   0    \quad \text{in} \ (0,\infty)\times \mathcal{F},\\
w(t,s,1) = (\partial_t \eta)(t,s) e_2  \quad t>0, \ s\in \mathcal{I},\\
w = 0  \quad t>0, \ y\in \Gamma_{-1},\\
\partial_{tt} \eta + A_1 \eta
	=-\mathbb{H}_{\eta_1^0}(w,q)+\widehat{G}_{\eta_1^0}(\xi,w), \quad t>0, \\
\end{array}\right.
\end{equation}
with the initial conditions
\begin{equation}\label{su4.3-CV-old}
\eta(0,\cdot)=\eta_1^0, \quad \partial_{t} \eta(0,\cdot)=\eta_2^0 \quad\mbox{ and }\quad  w(0, y)=w^0(y)\ov b(0,y) v^0(X(0,y))\  (y\in \mathcal{F}),
\end{equation}
where we have the following definitions:
\begin{multline}\label{bea4.2}
\widehat{F}_\alpha(\eta,w,q)
\ov
\nu\sum_{i,j,k} b_{\alpha i} \frac{\partial^2 a_{ik}}{\partial x_j^2}(X) w_k
+2\nu \sum_{i,j,k,\ell} b_{\alpha i} \frac{\partial a_{ik}}{\partial x_j}(X) \frac{\partial w_k}{\partial y_\ell} \frac{\partial Y_\ell}{\partial x_j}(X)
\\
+\nu \sum_{j,\ell,m}  \frac{\partial^2 w_\alpha}{\partial y_\ell\partial y_m} 
\left(\frac{\partial Y_\ell}{\partial x_j}(X)\frac{\partial Y_m}{\partial x_j}(X) -\delta_{\ell,j}\delta_{m,j}\right)
+\nu \sum_{j,\ell} \frac{\partial w_\alpha}{\partial y_\ell} \frac{\partial^2 Y_\ell}{\partial x_j^2}(X)
\\
-\sum_{k,i} \frac{\partial q}{\partial y_k}
 \left(\det(\nabla X)  \frac{\partial Y_\alpha}{\partial x_i}(X) \frac{\partial Y_k}{\partial x_i}(X)-\delta_{\alpha,i}\delta_{k,i}\right)
\\
-\sum_{i,j,k,m} b_{\alpha i} \frac{\partial a_{ik}}{\partial x_j}(X) a_{j m}(X) w_k w_m
-\frac{1}{\det(\nabla X)}  \left[(w\cdot \nabla) w\right]_\alpha
\\
-\left[b(\partial_t a)(X) w\right]_\alpha
-\left[(\nabla w)(\partial_tY)(X)\right]_\alpha,
\end{multline}
\begin{multline}\label{bea4.3}
\widehat{G}_{\eta_1^0}(\eta,w)(t,s)
=\nu M\Bigg\{ 
2\sum_{k,\ell} \left[\delta_{2,k}\delta_{2,\ell}-a_{2k}(X)\frac{\partial Y_\ell}{\partial x_2}(X)\right] \frac{\partial w_k}{\partial y_\ell} +
\partial_s (\eta-\eta_1^0)\left(\frac{\partial w_2}{\partial y_1} +\frac{\partial w_1}{\partial y_2}\right)
\\
+\partial_s \eta
\left[
  \sum_{k,\ell}  \left(a_{2k}(X)\frac{\partial Y_\ell}{\partial x_1}(X)- \delta_{2,k}\delta_{1,\ell }\right) \frac{\partial w_k}{\partial y_\ell} 
+
   \sum_{k,\ell}  \left(a_{1k}(X)\frac{\partial Y_\ell}{\partial x_2}(X)-\delta_{1,k}\delta_{2,\ell} \right)\frac{\partial w_k}{\partial y_\ell} 
\right]
\\
+ \sum_k \left(\partial_s \eta\left[\frac{\partial a_{2k}}{\partial x_1}(X)+\frac{\partial a_{1k}}{\partial x_2}(X)\right]-2\frac{\partial a_{2k}}{\partial x_2}(X) \right) w_k
\Bigg\}(t,s,1+\eta_1^0(s)).
\end{multline}
Moreover, we recall that
\begin{equation}\label{bea8.0}
\mathbb{H}_{\eta_1^0}(w,q)(t,s)
=
M\left\{ (1+|\partial_s \eta_1^0|^2)^{1/2} \left[\mathbb{T}(v,p)n\right](t,s,1+\eta_1^0(s))\cdot e_2 \right\}.
\end{equation}

\subsection{The linear system}
From the previous section, and in particular from system \eqref{tak2.3-CV2}-\eqref{su4.3-CV-old}, we are led to consider the linear system \eqref{gev0.3}-\eqref{bis} written in the fixed domain
$\mathcal{F}$ (defined by \eqref{F0}). We introduce the notation 
\begin{equation}\label{gev2.4bis}
\mathbb{C}^+\ov \left\{\lambda\in \mathbb{C} \ ; \ \Re(\lambda)\geq 0 \right\}. 
\end{equation}
\begin{equation}\label{gev2.4}
\mathbb{C}_{\alpha}^+\ov \left\{\lambda\in \mathbb{C}^+ \ ; \  |\lambda|>\alpha\right\}. 
\end{equation}

Let us consider the following functional spaces
\begin{equation}\label{gev9.2-bis}
\mathbf{V}^\theta(\mathcal{F}) \ov  \left\{f \in \mathbf{H}^\theta(\mathcal{F}) \ ; \ \div f=0 \right\},
\end{equation}
\begin{equation}\label{gev9.2}
\mathbf{V}^\theta_{n}(\mathcal{F}) \ov  
\left\{f \in \mathbf{H}^\theta(\mathcal{F}) \ ; \ \div f=0, \quad f \cdot n =0 \quad \text{on} \ \partial\mathcal{F} \right\}
\quad (\theta\in [0,1/2)),
\end{equation}
\begin{equation}\label{gev9.3}
\mathbf{V}^\theta_{n}(\mathcal{F}) \ov  
\left\{f \in \mathbf{H}^\theta(\mathcal{F}) \ ; \ \div f=0, \quad  f =0 \quad \text{on} \ \partial\mathcal{F} \right\}
\quad (\theta\in (1/2,1]),
\end{equation}
\begin{equation}\label{gev3.2bis}
\mathbf{V}^\theta(\partial \mathcal{F}) \ov  
\left\{f \in \mathbf{H}^\theta(\partial \mathcal{F}) \ ; \ 
\int_{\partial \mathcal{F}} f \cdot n \ d\gamma=0  \right\}
\quad (\theta\geq 0).
\end{equation}

We introduce the operator $\Lambda : L^2(\mathcal{I}) \to {\bf L}^2(\partial\mathcal{F})$ defined by
\begin{equation}\label{ws9.1}
\begin{split}
(\Lambda \eta)(y) &= \left(M \eta(s)\right) e_2\quad \mbox{if} \quad  y=(s,1+\eta_1^0(s))\in \Gamma_{\eta_1^0},\\
 (\Lambda \eta)(y)&=0\quad\quad\quad\, \mbox{if} \quad y\in \Gamma_{-1}.
 \end{split}
\end{equation}
The adjoint $\Lambda^*: {\bf L}^2(\partial \mathcal{F})\to L^2(\mathcal{I})$ of $ \Lambda$ is given by
\begin{equation}\label{ws9.3}
(\Lambda^* v)(s)=M\left((1+|\partial_s \eta_1^0(s)|^2)^{1/2} v(s, 1+\eta_1^0(s))\cdot e_2\right).
\end{equation}
Since $\eta_1^0\in W^{7,\infty}(\mathcal{I})$, then for any $\theta\in [0,4]$,
\begin{equation}\label{bea0.0}
\Lambda(H^\theta(\mathcal{I}))\subset {\bf V}^\theta(\partial \mathcal{F})
\end{equation}
and
\begin{equation}\label{gev9.5}
\Lambda^*({\bf H}^\theta(\partial \mathcal{F}))\subset \mathcal{D}(A_1^{\theta/4}).
\end{equation}
In particular
\begin{equation}\label{gev9.6}
\|\Lambda\eta\|_{\mathbf{H}^\theta(\partial \mathcal{F})} \geq c(\theta) \|A_1^{\theta/4} \eta\|_{\mathcal{H}_{\mathcal{S}}} \quad (\eta \in \mathcal{D}(A_1^{\theta/4})).
\end{equation}

We can also define the Stokes operator 
\begin{equation}\label{gev9.4}
\mathcal{D}(\mathbb{A})\ov \mathbf{V}^1_{n}(\mathcal{F})\cap \mathbf{H}^2(\mathcal{F}),
\quad
\mathbb{A}\ov \mathbb{P}\Delta : \mathcal{D}(\mathbb{A}) \to \mathbf{V}^0_{n}(\mathcal{F}),
\end{equation}
where $\mathbb{P} : \mathbf{L}^2(\mathcal{F}) \to \mathbf{V}^0_{n}(\mathcal{F})$ is the Leray projection operator.

We consider the space ${\bf L}^2(\mathcal{F})\times \mathcal{D}(A_1^{1/2}) \times \mathcal{H}_{\mathcal{S}} $ equipped with the scalar product:
$$
\left\langle 
\left[ w^{(1)}, \eta_1^{(1)}, \eta_2^{(1)}  \right], 
\left[ w^{(2)}, \eta_1^{(2)}, \eta_2^{(2)}  \right]
\right\rangle 
=\int_{\mathcal{F}} w^{(1)} \cdot w^{(2)} \ dy + \Big( A_1^{1/2}\eta_{1}^{(1)} , A_1^{1/2} \eta_{1}^{(2)} \Big)_{\mathcal{H}_{\mathcal{S}}}
+ \Big( \eta_{2}^{(1)} , \eta_{2}^{(2)} \Big)_{\mathcal{H}_{\mathcal{S}}},
$$  
and we introduce the following spaces:
\begin{equation}\label{gev7.4}
\mathcal{H} \ov \left\{\left[w,\eta_1,\eta_2\right]\in {\bf L}^2(\mathcal{F})\times \mathcal{D}(A_1^{1/2}) \times  \mathcal{H}_{\mathcal{S}} 
\ ;\ w\cdot n = (\Lambda \eta_2)\cdot n  \ \text{on}
\ \partial\mathcal{F},\ \div\, w =0 \; \text{in} \ \mathcal{F} \right\},
\end{equation}
$$
\mathcal{V} \ov \left\{\left[w,\eta_1,\eta_2\right]\in {\bf H}^1(\mathcal{F})\times \mathcal{D}(A_1^{3/4}) \times  \mathcal{D}(A_1^{1/4}) 
 \ ; \ w= \Lambda \eta_2  \ \text{on} \ \partial\mathcal{F}, 
\ \div\, w =0 \; \text{in} \ \mathcal{F}\right\}.
$$
We denote by $P_0$ the orthogonal projection from ${\bf L}^2(\mathcal{F})\times \mathcal{D}(A_1^{1/2}) \times  \mathcal{H}_{\mathcal{S}}$ onto $\mathcal{H}$. 
We have the following regularity result on $P_0$ (see \cite{plat}): 
\begin{Lemma}
For any $\theta\in [0,1]$,
\begin{equation}
P_0 \in \mathcal{L}({\bf H}^\theta(\mathcal{F})\times \mathcal{D}(A_1^{1/2+\theta/4})\times \mathcal{D}(A_1^{\theta/4})),\label{RegTildeP}
\end{equation}
and 
\begin{equation}
P_0 \in \mathcal{L}({\bf L}^2(\mathcal{F})\times \mathcal{D}(A_1^{3/8})\times \mathcal{D}(A_1^{1/8})').\label{RegTildeP-0}
\end{equation}
\end{Lemma}

We now define the linear operator 
$A_0 : \mathcal{D}(A_0)\subset \mathcal{H}\to \mathcal{H}$:
\begin{equation}\label{fs3.4}
\mathcal{D}(A_0)\ov  \mathcal{V} \cap  \Big [{\bf H}^2(\mathcal{F})\times \mathcal{D}(A_1)\times \mathcal{D}(A_1^{1/2})\Big],
\end{equation}
and for $\begin{bmatrix} w,\eta_1,\eta_2 \end{bmatrix}\in \mathcal{D}(A_0)$, we set
\begin{equation}\label{fs3.2}
\widetilde{A}_0\begin{bmatrix} w \\ \eta_1 \\ \eta_2\end{bmatrix} \ov
\begin{bmatrix}
 \Delta w \medskip\\ \medskip
\eta_2 \displaystyle \\ \medskip
 \displaystyle -A_1\eta_{1}-\Lambda^*(2 D(w)n)
\end{bmatrix}
\end{equation}
and
\begin{equation}\label{fs3.3}
A_0\ov P_0  \widetilde{A}_0.
\end{equation}

By using the above operators, we can rewrite the linear system \eqref{gev0.3}, \label{gev0.3-bis}
as follows
\begin{equation}\label{gev0.3-ter}
\frac{d}{dt} \begin{bmatrix} w\\ \eta \\ \partial_t \eta \end{bmatrix} 
= A_0 \begin{bmatrix} w \\ \eta \\ \partial_t \eta \end{bmatrix} + P_0 \begin{bmatrix} F\\ 0\\ G \end{bmatrix},
\quad
\begin{bmatrix} w\\ \eta \\ \partial_t \eta \end{bmatrix}(0)=\begin{bmatrix} w^0 \\ \eta_1^0 \\ \eta_2^0 \end{bmatrix}.
\end{equation}

We also recall the following result (see \cite[Proposition 3.4, Proposition 3.5 and Remark 3.6]{MR3619065}).
\begin{Proposition}\label{AStrongContSG}
The operator $A_0$ defined by \eqref{fs3.4}--\eqref{fs3.3} has compact resolvents, 
it is the infinitesimal generator of a strongly continuous semigroup of contractions on $\mathcal{H}$ and it is exponentially stable on $\mathcal{H}$.
\end{Proposition}
We have also the following result (see \cite[Proposition 3.8]{MR3619065}).
\begin{Proposition}\label{PropA}  
For $\theta\in [0,1]$, the following equalities hold
\begin{equation}
\mathcal{D}((-A_0)^\theta)= \left[{\bf H}^{2\theta}(\mathcal{F})\times \mathcal{D}(A_1^{1/2+\theta/2})\times \mathcal{D}(A_1^{\theta/2})\right]\cap \mathcal{H} \quad \mbox{ if }\;\theta\in \left(0,1/4\right), \label{Domainfraceps1}
\end{equation}
\begin{multline}
\mathcal{D}((-A_0)^\theta)=\left\{[w,\eta_1,\eta_2]\in \left[{\bf H}^{2\theta}(\mathcal{F})\times \mathcal{D}(A_1^{1/2+\theta/2})\times \mathcal{D}(A_1^{\theta/2})\right]\cap \mathcal{H}  \,;\,  
w= \Lambda \eta_2 \ \text{on} \ \partial \mathcal{F} \right\}
\\ 
\mbox{ if }\;\theta\in (1/4, 1].\label{Domainfraceps2}
\end{multline}
\end{Proposition}

One of the main goals of this article is to show the following result:
\begin{Theorem}\label{CorMain}
There exists $C>0$ such that for all $\lambda \in \mathbb{C}^+$
\begin{equation}\label{ResEstA0-1}
|\lambda|^{1/2} \left\| (\lambda I-A_0)^{-1}\right\|_{\mathcal{L}(\mathcal{H})}
\leq C.
\end{equation}
Moreover, there exists a constant $C>0$ such that for all $\lambda \in \mathbb{C}^+$ 
\begin{multline}\label{ResEstA0-3}
\left\| (\lambda I-A_0)^{-1}
z
\right\|_{{\bf H}^2(\mathcal{F})\times \mathcal{D}(A_1^{7/8})\times \mathcal{D}(A_1^{3/8})}
+|\lambda|\left\| (\lambda I-A_0)^{-1}
z
\right\|_{{\bf L}^2(\mathcal{F})\times \mathcal{D}(A_1^{3/8})\times \mathcal{D}(A_1^{1/8})'}
\\
\leq C\left\|
z%
\right\|_{{\bf L}^2(\mathcal{F})\times \mathcal{D}(A_1^{5/8})\times \mathcal{D}(A_1^{1/8})}
\\
\left(
z
\in \mathcal{H}\cap \left({\bf L}^2(\mathcal{F})\times \mathcal{D}(A_1^{5/8})\times \mathcal{D}(A_1^{1/8})\right)\right).
\end{multline}
\end{Theorem}
Using the above theorem and Theorem 5.1 in \cite{plat}, we deduce \cref{T04}.
\section{Definition and properties of some operators}\label{sec_op}
This section is devoted to the introduction of several operators that are used to prove the resolvent estimates in \cref{CorMain}.
In this section we assume $\eta_1^0\in W^{4,\infty}(\mathcal{I})$. It implies in particular that the domain $\mathcal{F}$ is of class $C^{3,1}$. 

For all $\lambda\in \mathbb{C}^{+}$, we define the solution $(w_\eta, q_\eta)$ (that depends on $\lambda$) of
\begin{equation}\label{gev0.6}
\left\{\begin{array}{rl}
\lambda w_\eta -\div \mathbb{T}(w_\eta,q_\eta) = 0 & \text{in} \ \mathcal{F},\\
\div w_\eta =   0  & \text{in} \ \mathcal{F},\\
w_\eta = \Lambda\eta  & \text{on} \  \partial \mathcal{F},
\end{array}\right.
\end{equation}
where $\Lambda$ is defined by \eqref{ws9.1}. The above problem is well-posed (see, for instance, \cite[Proposition 4.4]{plat}) and if we define the operators
\begin{equation}\label{bea0.1}
W_{\lambda} \eta \ov w_\eta, \quad Q_{\lambda} \eta \ov q_\eta,
\end{equation}
since $\mathcal{F}$ is of class $C^{3,1}$, we have
\begin{equation}\label{np0004}
W_{\lambda} \in \mathcal{L}(\mathcal{D}(A_1^{7/8}),\mathbf{H}^4(\mathcal{F}))
	\cap \mathcal{L}(\mathcal{D}(A_1^{1/8}),\mathbf{H}^1(\mathcal{F}))
	\cap \mathcal{L}(\mathcal{D}(A_1^{1/8})',\mathbf{L}^2(\mathcal{F}))
\end{equation}
and
\begin{equation}\label{np0005}
Q_{\lambda} \in \mathcal{L}(\mathcal{D}(A_1^{3/8}),H^1(\mathcal{F})/\mathbb{R}).
\end{equation}

We also define the operator
$$
L_{\lambda}\in \mathcal{L}(\mathcal{D}(A_1^{3/8}),\mathcal{D}(A_1^{1/8}))
$$ 
by
\begin{equation}\label{DefG}
\displaystyle L_{\lambda} \eta \ov  \Lambda^* \left\{\mathbb{T}(w_\eta,q_\eta)n_{|\partial\mathcal{F}}\right\}.
\end{equation}

We decompose $L_{\lambda}$ with the operators
$$
K_{\lambda}\in \mathcal{L}(\mathcal{D}(A_1^{1/8})',\mathcal{D}(A_1^{1/8})),
\quad 
G_{\lambda}\in \mathcal{L}(\mathcal{D}(A_1^{1/8}),\mathcal{D}(A_1^{1/8})') \cap \mathcal{L}(\mathcal{D}(A_1^{3/8}),\mathcal{D}(A_1^{1/8}))
$$
defined by
\begin{equation}\label{np0034}
\langle K_{\lambda}\eta, \zeta\rangle_{\mathcal{D}(A_1^{1/8}),\mathcal{D}(A_1^{1/8})'} \ov \int_\mathcal{F} w_\eta\cdot \overline{w_\zeta} dy
\end{equation}
and
$$
\langle G_{\lambda}\eta, \zeta\rangle_{\mathcal{D}(A_1^{1/8})',\mathcal{D}(A_1^{1/8})} \ov 2 \nu \int_\mathcal{F} D w_\eta : D \overline{w_\zeta} dy
=2\nu \int_{0}^L \Lambda^*((D w_\eta)n) \ \overline{\zeta} \ ds-  \nu \int_\mathcal{F} \Delta w_\eta \cdot  \overline{w_\zeta} dy.
$$
(The second relation holds if $\eta\in \mathcal{D}(A_1^{3/8})$).

The operators $K_{\lambda}$ and $G_{\lambda}$ are related to the operator $L_{\lambda}$ defined by \eqref{DefG}:
multiplying \eqref{gev0.6} by $\overline{w_{\zeta}}$ and integrating by part, we deduce that 
\begin{equation}\label{gev1.2}
L_{\lambda}=\lambda K_{\lambda}+G_{\lambda}.
\end{equation}
We recall the following result (see Proposition 3.1 in \cite{plat}):
\begin{Proposition}\label{P04}
The operators $K_{\lambda}$ and $G_{\lambda}$ defined above are positive and self-adjoint. Moreover there exist 
$0<\rho_1<\rho_2$ such that for any $\lambda$ such that $\Re \lambda>0$, we have 
\begin{equation}\label{eq28-11-1a}
\rho_1 \|A_{1}^{1/8}\eta\|^2_{\mathcal{H}_{\mathcal{S}}}
\leq 
\langle G_{\lambda}\eta, \eta\rangle_{\mathcal{D}(A_1^{1/8})',\mathcal{D}(A_1^{1/8})}
\leq \rho_2 \left(\|A_{1}^{1/8}\eta\|^2_{\mathcal{H}_{\mathcal{S}}}+|\lambda|\|A_1^{-1/8}\eta\|_{\mathcal{H}_{\mathcal{S}}}^2\right)
 \quad (\eta\in \mathcal{D}(A_1^{1/8})),
\end{equation}
\begin{equation}\label{eq28-11-1b}
0 \leq 
\langle K_{\lambda}\eta, \eta\rangle_{\mathcal{D}(A_1^{1/8}),\mathcal{D}(A_1^{1/8})'} 
\leq \rho_2 \|A_1^{-1/8}\eta\|_{\mathcal{H}_{\mathcal{S}}}^2
\quad (\eta\in \mathcal{D}(A_1^{1/8})').
\end{equation}
\end{Proposition}

Note that we have 
\begin{equation}\label{DefK}
K_{\lambda}\eta=-\Lambda^* \{\mathbb{T}(\varphi_\eta,\pi_\eta)n|_{\partial \mathcal{F}}\}
\end{equation}
where 
\begin{equation}\label{gev1.3DefK}
\left\{\begin{array}{rl}
\overline{\lambda} \varphi_\eta -\div \mathbb{T}(\varphi_\eta,\pi_\eta) = W_{\lambda}\eta & \text{in}\ \mathcal{F},\\
\div \varphi_\eta =   0  & \text{in}\ \mathcal{F},\\
\varphi_\eta = 0  & \text{on}\ \partial \mathcal{F},
\end{array}\right.
\end{equation}
and where $W_{\lambda}$ is defined by \eqref{bea0.1}.

Next, we define an important operator in what follows:
\begin{equation}\label{gev9.8}
V_{\lambda}=\lambda^2 I+\lambda L_{\lambda}+A_1
=\lambda^2 (I+K_{\lambda})+ \lambda G_{\lambda}+ A_1,
\end{equation}
and an ``approximation'':
\begin{equation}\label{gev2.3}
\widetilde{V}_{\lambda}\ov \lambda^2 (I+K_{\lambda}) + 2\rho \lambda A_1^{1/4}+ A_1,
\end{equation}
where $\rho>0$ is a constant to be fixed later.

Let us consider
\begin{equation}\label{gev0.5}
\left\{\begin{array}{rl}
\lambda \widehat v -\div \mathbb{T}(\widehat v,\widehat p) = f& \text{in}\ \mathcal{F},\\
\div \widehat  v =   0& \text{in}\ \mathcal{F},\\
\widehat  v = 0 & \text{on}\ \partial \mathcal{F}.
\end{array}\right.
\end{equation}

\begin{Proposition}\label{prop23novBis}
Let $\gamma\in [0,1/4)$, $\theta\in [\gamma,1]$. 
\begin{enumerate}
\item There exists $C>0$ such that for any $f\in {\bf H}^{2\gamma}(\mathcal{F})$
and for any $\lambda \in \mathbb{C}^+_0$, the solution $(\widehat v,\widehat p)$ of \eqref{gev0.5} satisfies
\begin{equation}\label{EstStokeslambdanh23novBis}
\|\widehat v\|_{{\bf H}^{2\theta}(\mathcal{F})}\leq C |\lambda|^{\theta-\gamma-1} \|f\|_{{\bf H}^{2\gamma}(\mathcal{F})}.
\end{equation}
\item There exists $C>0$ such that for any $f\in {\bf H}^{2\theta}(\mathcal{F})$
and for any $\lambda \in \mathbb{C}^+_0$, the solution $(\widehat v,\widehat p)$ of \eqref{gev0.5} satisfies
\begin{equation}\label{EstStokeslambdanh23nov}
\|\widehat v\|_{{\bf H}^{2+2\theta}(\mathcal{F})}+\|\widehat p\|_{{\bf H}^{1+2\theta}(\mathcal{F})}
	\leq C \left(|\lambda|^{\theta-\gamma} \|f\|_{{\bf H}^{2\gamma}(\mathcal{F})}+\|f\|_{{\bf H}^{2\theta}(\mathcal{F})}\right).
\end{equation}
\end{enumerate}
\end{Proposition}
\begin{proof}
Using that the Stokes operator $\mathbb{A}$ (defined by \eqref{gev9.4}) is the infinitesimal generator of an analytic semigroup and that
$\mathbb{C}^+_0\subset \rho(\mathbb{A})$, we have the existence of a constant $C$ such that 
$$
\|(-\mathbb{A})^{\alpha} (\lambda I-\mathbb{A})^{-1} g \|_{{\bf L}^{2}(\mathcal{F})}
\\
\leq C|\lambda|^{\alpha-1} \|g \|_{{\bf L}^{2}(\mathcal{F})} 
\quad (g\in {\bf V}_{n}^{0}(\mathcal{F}), \ \lambda\in \mathbb{C}^+_0, \ \alpha \in [0,1]).
$$
Using that for $\gamma\in [0,1/4)$, $\mathbb{P}\in \mathcal{L}({\bf H}^{2\gamma}(\mathcal{F}), \mathcal{D}((-\mathbb{A})^{\gamma}))$ (see \cite[Section 2.1]{BADRA-SICON-2009}), we have
$$
\|(-\mathbb{A})^{\gamma} \mathbb{P} f \|_{{\bf L}^{2}(\mathcal{F})}
\leq C \|f\|_{{\bf H}^{2\gamma}(\mathcal{F})} \quad (\gamma \in [0,1/4), \ f\in {\bf H}^{2\gamma}(\mathcal{F})).
$$
Gathering the two above estimates with the fact that
$\mathcal{D}((-\mathbb{A})^{\theta}) \subset {\bf H}^{2\theta}(\mathcal{F})$, we can deduce
$$
\|\widehat  v\|_{{\bf H}^{2\theta}(\mathcal{F})}
\leq C\left\|(-\mathbb{A})^{\theta}(\lambda I-\mathbb{A})^{-1} \mathbb{P} f\right\|_{{\bf L}^{2}(\mathcal{F})}
\leq C|\lambda|^{\theta-\gamma-1} \|f\|_{{\bf H}^{2\gamma}(\mathcal{F})}.
$$
For the second estimate, we use the following classical estimate for Stokes system:
$$
\|\widehat v\|_{{\bf H}^{2+2\theta}(\mathcal{F})}+\|\widehat p\|_{{\bf H}^{1+2\theta}(\mathcal{F})}
	\leq C\left(|\lambda|\|\widehat v\|_{{\bf H}^{2\theta}(\mathcal{F})}+\|f\|_{{\bf H}^{2\theta}(\mathcal{F})}\right),
$$
and we combine it with  \eqref{EstStokeslambdanh23novBis}.
\end{proof}

Using the above proposition, we can define the following operator
\begin{equation}\label{np0056}
\mathcal{T}_{\lambda} \in \mathcal{L}({\bf L}^2(\mathcal{F}),\mathcal{D}(A_1^{1/8})),
\quad
\mathcal{T}_{\lambda}f\ov - \Lambda^*\left\{\mathbb{T}(\widehat v,\widehat p)n_{|\partial\mathcal{F}}\right\}.
\end{equation}
We have in particular that the norm of 
$\mathcal{T}_{\lambda}$ in $\mathcal{L}({\bf L}^2(\mathcal{F}),\mathcal{D}(A_1^{1/8}))$ is independent of $\lambda$.

\begin{Proposition}\label{P01}
For $\theta\in [0,1]$, $\varepsilon\in (0,1/4)$ and $\lambda \in \mathbb{C}^+_0$, the operators $W_{\lambda}$ and $Q_{\lambda}$ defined by \eqref{bea0.1} satisfy
\begin{equation}\label{eq23-06-2017-1-0}
\|W_{\lambda} \eta\|_{{\bf H}^{2\theta}(\mathcal{F})}\leq C\|A_1^{\theta/2-1/8}\eta\|_{\mathcal{H}_{\mathcal{S}}}\quad (\theta< 1/4),
\end{equation}
\begin{equation}\label{eq23-06-2017-1}
\|W_{\lambda}\eta\|_{{\bf H}^{2\theta}(\mathcal{F})}\leq C\left(\|A_1^{\theta/2-1/8}\eta\|_{\mathcal{H}_{\mathcal{S}}}+|\lambda|^{\theta}\|A_1^{-1/8}\eta\|_{\mathcal{H}_{\mathcal{S}}}\right)\quad (\theta\geq 1/4),
\end{equation}
\begin{equation}\label{eq23-06-2017-2}
\|W_{\lambda}\eta\|_{{\bf H}^{2\theta}(\mathcal{F})}\leq C\left(\|A_1^{\theta/2-1/8}\eta\|_{\mathcal{H}_{\mathcal{S}}}+|\lambda|^{\theta-1/4+\varepsilon}\|\eta\|_{\mathcal{H}_{\mathcal{S}}}\right),\quad \theta\geq 1/4-\varepsilon,
\end{equation}
\begin{equation}\label{eq23-06-2017-1Bis}
\|W_{\lambda}\eta\|_{{\bf H}^{2+2\theta}(\mathcal{F})}+\|Q_{\lambda}\eta\|_{H^{1+2\theta}(\mathcal{F})}
\leq C\left(\|A_1^{\theta/2+3/8}\eta\|_{\mathcal{H}_{\mathcal{S}}}+|\lambda|^{1+\theta}\|A_1^{-1/8}\eta\|_{\mathcal{H}_{\mathcal{S}}}\right),
\end{equation}
\begin{equation}\label{eq23-06-2017-2Bis}
\|W_{\lambda}\eta\|_{{\bf H}^{2+2\theta}(\mathcal{F})}+\|Q_{\lambda}\eta\|_{H^{1+2\theta}(\mathcal{F})}
\leq C\left(\|A_1^{\theta/2+3/8}\eta\|_{\mathcal{H}_{\mathcal{S}}}+|\lambda|^{3/4+\varepsilon+\theta}\|\eta\|_{\mathcal{H}_{\mathcal{S}}}\right).
\end{equation}
\end{Proposition}
\begin{proof}
We write
$$
W_{\lambda}\eta = W_{0}\eta +z_\eta, \quad Q_{\lambda}\eta = Q_{0}\eta +\zeta_\eta,
$$
with
\begin{equation}\label{gev1.6}
\left\{\begin{array}{c}
\lambda z_\eta -\div \mathbb{T}(z_\eta,\zeta_\eta) = -\lambda W_{0}\eta\quad \text{in} \ \mathcal{F},\\
\div z_\eta =   0  \quad \text{in} \ \mathcal{F},\\
z_\eta = 0 \quad \text{on} \ \partial \mathcal{F}.
\end{array}\right.
\end{equation}
Using \eqref{np0004}, there exists a positive constant $C$ such that
\begin{equation}\label{eq2017-1}
\|W_0 \eta\|_{{\bf H}^{2\theta}(\mathcal{F})} 
\leq 
C \|A_1^{\theta/2-1/8} \eta\|_{\mathcal{H}_{\mathcal{S}}} \quad (\theta\in [0,2], \ \eta \in \mathcal{D}(A_1^{\theta/2-1/8})).
\end{equation}
Combining the above relation with \eqref{EstStokeslambdanh23novBis} we deduce the following relations:
$$
\|z_{\eta}\|_{{\bf H}^{2\theta}(\mathcal{F})}
\leq C |\lambda|^{\theta}\|A_1^{-1/8} \eta\|_{\mathcal{H}_{\mathcal{S}}},
$$
$$
\|z_{\eta}\|_{{\bf H}^{2\theta}(\mathcal{F})}
\leq C \|A_1^{\theta/2-1/8} \eta\|_{\mathcal{H}_{\mathcal{S}}} \quad \text{if} \ \theta\in [0,1/4),
$$
$$
\|z_{\eta}\|_{{\bf H}^{2\theta}(\mathcal{F})}
\leq C |\lambda|^{\theta-1/4+\varepsilon}\|\eta\|_{\mathcal{H}_{\mathcal{S}}}
\quad (\theta\geq 1/4-\varepsilon).
$$
Then \eqref{eq23-06-2017-1-0}, \eqref{eq23-06-2017-1} and \eqref{eq23-06-2017-2} follow by combining the above inequalities with \eqref{eq2017-1}.

From \eqref{EstStokeslambdanh23nov} we deduce
\begin{equation}\label{np0006}
\|z_{\eta}\|_{{\bf H}^{2+2\theta}(\mathcal{F})}+\|\zeta_{\eta}\|_{H^{1+2\theta}(\mathcal{F})}
\leq 
C \left(|\lambda|^{3/4+\varepsilon+\theta}\|W_{0}\eta\|_{{\bf H}^{1/2-2\varepsilon}(\mathcal{F})}+|\lambda|\|W_{0}\eta\|_{{\bf H}^{2\theta}(\mathcal{F})}\right),
\end{equation}
and
\begin{equation}\label{np0007}
\|z_{\eta}\|_{{\bf H}^{2+2\theta}(\mathcal{F})}+\|\zeta_{\eta}\|_{H^{1+2\theta}(\mathcal{F})}
\leq 
C \left(|\lambda|^{1+\theta}\|W_{0}\eta\|_{{\bf L}^{2}(\mathcal{F})}+|\lambda|\|W_{0}\eta\|_{{\bf H}^{2\theta}(\mathcal{F})}\right).
\end{equation}
Moreover, from \eqref{eq2017-1}, we have
\begin{multline*}
|\lambda|\|W_{0}\eta\|_{{\bf H}^{2\theta}(\mathcal{F})}
\leq C\left(|\lambda|^{1+\theta}\|W_{0}\eta\|_{{\bf L}^{2}(\mathcal{F})}\right)^{\frac{1}{1+\theta}}\left(\|W_{0}\eta\|_{{\bf H}^{2+2\theta}(\mathcal{F})}\right)^{\frac{\theta}{1+\theta}}
\\
\leq C\left(\|W_{0}\eta\|_{{\bf H}^{2+2\theta}(\mathcal{F})}+|\lambda|^{1+\theta}\|W_{0}\eta\|_{{\bf L}^{2}(\mathcal{F})}\right)
\leq C\left(\|A_1^{\theta/2+3/8}\eta\|_{\mathcal{H}_{\mathcal{S}}}+|\lambda|^{1+\theta}\|A_1^{-1/8}\eta\|_{\mathcal{H}_{\mathcal{S}}}\right),
\end{multline*}
and
\begin{multline*}
|\lambda|\|W_{0}\eta\|_{{\bf H}^{2\theta}(\mathcal{F})}\leq  C\left(|\lambda|^{1+\theta-1/4+\varepsilon}\|W_{0}\eta\|_{{\bf H}^{1/2-2\varepsilon}(\mathcal{F})}\right)^{\frac{1}{1+\theta-1/4+\varepsilon}}\left(\|W_{0}\eta\|_{{\bf H}^{2+2\theta}(\mathcal{F})}\right)^{\frac{\theta-1/4+\varepsilon}{1+\theta-1/4+\varepsilon}}\\
\leq C\left(\|W_{0}\eta\|_{{\bf H}^{2+2\theta}(\mathcal{F})}+|\lambda|^{3/4+\varepsilon+\theta}\|W_{0}\eta\|_{{\bf H}^{1/2-2\varepsilon}(\mathcal{F})}\right)
\leq C\left(\|A_1^{\theta/2+3/8}\eta\|_{\mathcal{H}_{\mathcal{S}}}+|\lambda|^{3/4+\varepsilon+\theta}\|\eta\|_{\mathcal{H}_{\mathcal{S}}}\right).
\end{multline*}
Combining the above estimates with \eqref{np0006} and \eqref{np0007}, we deduce \eqref{eq23-06-2017-1Bis} and \eqref{eq23-06-2017-2Bis}.
\end{proof}

\begin{Proposition}\label{P03}
Let $\theta\in [1/4,1/2)$ and $\lambda \in \mathbb{C}^+_0$, then
\begin{equation} \label{eq25112018}
\|A_1^{\theta/2} K_{\lambda}\eta\|_{\mathcal{H}_S}\leq C\|A_1^{\theta/2-1/4} \eta\|_{\mathcal{H}_S}.
\end{equation}
Let $\varepsilon\in (0,1/4)$, $\lambda \in \mathbb{C}^+_0$ and $\theta\in [1/2,2]$, then 
\begin{equation} \label{eq25112018-1}
\|A_1^{\theta/2} K_{\lambda}\eta\|_{\mathcal{H}_S}
	\leq C\left(\|A_1^{\theta/2-1/4} \eta\|_{\mathcal{H}_S}+|\lambda|^{\theta-1/2+\varepsilon}\|\eta\|_{\mathcal{H}_S}\right).
\end{equation}
\end{Proposition}
\begin{proof}
From \eqref{DefK} and properties on the trace operator and on $\Lambda^*$, 
\begin{equation}\label{np0008}
\|A_1^{\theta/2} K_{\lambda}\eta\|_{\mathcal{H}_S}\leq C\left(\|\varphi_\eta\|_{{\bf H}^{3/2+2\theta}(\mathcal{F})}+\|\pi_\eta\|_{H^{1/2+2\theta}(\mathcal{F})}\right).
\end{equation}
Using \eqref{gev1.3DefK}, \eqref{EstStokeslambdanh23nov} and \eqref{eq23-06-2017-1-0}, we deduce from the above estimate that
$$
\|A_1^{\theta/2} K_{\lambda}\eta\|_{\mathcal{H}_S}
\leq 	 C \|W_{\lambda}\eta\|_{{\bf H}^{2\theta-1/2}(\mathcal{F})}\leq C \|A_1^{\theta/2-1/4}\eta\|_{ \mathcal{H}_S}
$$
if $\theta\in [1/4,1/2)$.

On the other hand, if $\theta\in [1/2,2]$, using \eqref{gev1.3DefK}, \eqref{EstStokeslambdanh23nov}, \eqref{eq23-06-2017-1-0} and \eqref{eq23-06-2017-2}, we deduce from \eqref{np0008} that
\begin{multline*}
\|A_1^{\theta/2} K_{\lambda}\eta\|_{\mathcal{H}_S}
\leq C\left(|\lambda|^{\theta-1/2+\varepsilon} \|W_{\lambda}\eta\|_{{\bf H}^{1/2-2\varepsilon}(\mathcal{F})}+\|W_{\lambda}\eta\|_{{\bf H}^{2\theta-1/2}(\mathcal{F})}\right)
\\
\leq C\left(|\lambda|^{\theta-1/2+\varepsilon} \|\eta\|_{\mathcal{H}_S}+\|A_1^{\theta/2-1/4}\eta\|_{\mathcal{H}_S}\right).
\end{multline*}
\end{proof}

\begin{Proposition}\label{P02}
Assume $\alpha>0$. 
Let $\theta\in (-1/2,1/2)$, then
\begin{equation} \label{eq25112018-2}
\|A_1^{\theta/2} (I+K_{\lambda})^{-1}\eta\|_{\mathcal{H}_S}\leq C\|A_1^{\theta/2} \eta\|_{\mathcal{H}_S}
\quad (\lambda \in \mathbb{C}^+).
\end{equation}
Let $\varepsilon>0$ and $\theta\in [1/2,2]$, then 
\begin{equation} \label{eq25112018-4}
\|A_1^{\theta/2} (I+K_{\lambda})^{-1}\eta\|_{\mathcal{H}_S}\leq C\left(\|A_1^{\theta/2} \eta\|_{\mathcal{H}_S}+|\lambda|^{\theta-1/2+\varepsilon}\|\eta\|_{\mathcal{H}_S}\right)
\quad (\lambda \in \mathbb{C}^+_\alpha).
\end{equation}
\end{Proposition}
\begin{proof}
We write
$$
\zeta=(I+K_{\lambda})^{-1}\eta, \quad \zeta=\eta-K_{\lambda}\zeta.
$$
First, using the positivity of $K_{\lambda}$ stated in \eqref{eq28-11-1b}, we find
\begin{equation}\label{np0011}
\|\zeta\|_{\mathcal{H}_S}\leq \|\eta\|_{\mathcal{H}_S}.
\end{equation}
Moreover, we have the relation
\begin{equation}\label{np0010}
\|A_1^{\theta/2} \zeta\|_{\mathcal{H}_S}\leq \|A_1^{\theta/2} \eta\|_{\mathcal{H}_S}+\|A_1^{\theta/2} K_{\lambda}\zeta\|_{\mathcal{H}_S}.
\end{equation}
Assume first $\theta\in [1/4,1/2)$. Then combining \eqref{eq25112018} and \eqref{np0011} with \eqref{np0010}, we deduce \eqref{eq25112018-2}.
Interpolating \eqref{eq25112018-2} and \eqref{np0011}, we also deduce that \eqref{eq25112018-2} holds for $\theta \in [0,1/2)$. To prove \eqref{eq25112018-2} for $\theta \in (-1/2,0)$ we use a duality argument. First, for $\theta\in (0,1/2)$ \eqref{eq25112018-2} can be rewritten as
$$
\|A_1^{\theta/2} (I+K_{\lambda})^{-1}A_1^{-\theta/2}\|_{\mathcal{L}(\mathcal{H}_S)}<+\infty.
$$
Thus, noticing that $(A_1^{\theta/2} (I+K_{\lambda})^{-1}A_1^{-\theta/2})^*=A_1^{-\theta/2}(I+K_{\lambda})^{-1}A_1^{\theta/2}$ we also have
$$
\|A_1^{-\theta/2}(I+K_{\lambda})^{-1} A_1^{\theta/2}\|_{\mathcal{L}(\mathcal{H}_S)}<+\infty,
$$
which is equivalent to \eqref{eq25112018-2} for $\theta\in (-1/2,0)$.

Next, we assume $\theta\in [1/2,2]$ and $\varepsilon>0$. We deduce from  \eqref{np0010} and \eqref{eq25112018-1} that
\begin{equation}\label{np0012}
\|A_1^{\theta/2} \zeta\|_{\mathcal{H}_S}\leq \|A_1^{\theta/2} \eta\|_{\mathcal{H}_S}+C \left(\|A_1^{\theta/2-1/4}\zeta\|_{\mathcal{H}_S}+|\lambda|^{\theta-1/2+\varepsilon}\|\zeta\|_{\mathcal{H}_S}\right).
\end{equation}
If $\theta \in [1/2,1)$, then $\theta-1/2\in [0,1/2)$ and we can use \eqref{eq25112018-2} and \eqref{np0011} in the above relation to deduce \eqref{eq25112018-4}. If $\theta\in [1,3/2)$, then $\theta-1/2\in [1/2,1)$ and we can use \eqref{eq25112018-4} and  \eqref{np0011} in \eqref{np0012} to deduce
$$
\|A_1^{\theta/2} \zeta\|_{\mathcal{H}_S}\leq \|A_1^{\theta/2} \eta\|_{\mathcal{H}_S}+C \left(\|A_1^{\theta/2-1/4}\eta\|_{\mathcal{H}_S}+|\lambda|^{\theta-1+\varepsilon}\|\zeta\|_{\mathcal{H}_S}+|\lambda|^{\theta-1/2+\varepsilon}\|\zeta\|_{\mathcal{H}_S}\right).
$$
Since $|\lambda|>\alpha$ we have $|\lambda|^{\theta-1+\varepsilon}\leq \alpha^{-1/2}|\lambda|^{\theta-1/2+\varepsilon}$ and it yields \eqref{eq25112018-4}. We can then repeat the same argument for $\theta\in [3/2,2)$ and $\theta=2$.
\end{proof}

\section{Commutator estimates}\label{sec_com}
The aim of this section is to show the following result:
\begin{Lemma}\label{2017Commutateur1}
Assume $\eta_1^0\in W^{7,\infty}(\mathcal{I})$.
For $\varepsilon\in (0,1/4)$, there exists a constant $C>0$ such that for any $\lambda \in \mathbb{C}^+_0$, 
$$
\|[A_1^{3/8}, K_{\lambda}]\eta\|_{\mathcal{H}_{\mathcal{S}}} \leq C( |\lambda|^{-1}\|A^{1/2+\varepsilon}\eta\|_{\mathcal{H}_{\mathcal{S}}}+|\lambda|^{\varepsilon}\|\eta\|_{\mathcal{H}_{\mathcal{S}}}).
$$
\end{Lemma}

Here we have denoted by $[A,B]$ the commutator of $A$ and $B$: $[A,B]=AB-BA$.

\subsection{The system written in a domain with a flat boundary}\label{toflat}
We transform the systems \eqref{gev0.6}  and \eqref{gev1.3DefK}
written in $\mathcal{F}=\mathcal{F}_{\eta_1^0}$ into systems written in the domain 
$$
\mathcal{F}_{0}=\mathcal{I}\times (0,1).
$$
We use the change of variables
$$
\widetilde{X}  : \mathcal{F}_{0} \to \mathcal{F}, \quad (y_1,y_2)\mapsto \left(y_1, y_2(1+\eta_1^0(y_1))\right),
$$
$$
\widetilde{Y} : \mathcal{F} \to \mathcal{F}_{0}, \quad (x_1,x_2)\mapsto \left(x_1, \frac{x_2}{1+\eta_1^0(x_1)}\right).
$$

We write 
$$
\widetilde a\ov \Cof(\nabla \widetilde{Y})^*, \quad \widetilde b\ov \Cof(\nabla \widetilde{X})^*.
$$

We set
\begin{equation}\label{gev8.0}
\widetilde{w}(y)\ov \widetilde b(y) w(\widetilde{X}(y))\quad \mbox{ and }\quad
\widetilde{q}(y)\ov q(\widetilde{X}(y)),
\end{equation}
so that
\begin{equation}\label{gev8.1}
w(x)=\widetilde a(x) \widetilde{w}(\widetilde{Y}(x)) \quad \text{and} \quad 
q(x)=\widetilde{q}(\widetilde Y(x)).
\end{equation}
We set
\begin{multline}\label{gev8.5}
\left[\mathbb{L} \widetilde{w} \right]_\alpha
\ov \sum_{i,j,k} \widetilde b_{\alpha i} \frac{\partial^2 \widetilde a_{ik}}{\partial x_j^2}(\widetilde X) \widetilde{w}_k
+2\sum_{i,j,k,\ell} \widetilde b_{\alpha i} \frac{\partial \widetilde a_{ik}}{\partial x_j}(\widetilde X) \frac{\partial \widetilde{w}_k}{\partial y_\ell} \frac{\partial \widetilde Y_\ell}{\partial x_j}(\widetilde X)
\\
+\sum_{j,\ell,m}  \frac{\partial^2 \widetilde{w}_\alpha}{\partial y_\ell\partial y_m} \frac{\partial \widetilde Y_\ell}{\partial x_j}(\widetilde X)\frac{\partial \widetilde Y_m}{\partial x_j}(\widetilde X)
+\sum_{j,\ell} \frac{\partial \widetilde{w}_\alpha}{\partial y_\ell} \frac{\partial^2 \widetilde Y_\ell}{\partial x_j^2}(\widetilde X),
\end{multline}
\begin{equation}\label{gev8.9}
\left[\mathbb{G} \widetilde{q} \right]_\alpha
\ov \det(\nabla \widetilde X) \sum_{k,i} \frac{\partial \widetilde{q}}{\partial y_k} \frac{\partial \widetilde Y_\alpha}{\partial x_i}(\widetilde X) \frac{\partial \widetilde Y_k}{\partial x_i}(\widetilde X).
\end{equation}
Then some calculation yields 
\begin{equation}\label{formb}
\left[\widetilde b\Delta w(\widetilde X)\right]_\alpha=\left[\mathbb{L} \widetilde{w} \right]_\alpha,
\quad
\left[\widetilde b\nabla q(\widetilde X)\right]_\alpha=\left[\mathbb{G} \widetilde{q} \right]_\alpha.
\end{equation}
We recall the derivation of \eqref{formb} in \cref{sec_formulas}.

We now consider systems \eqref{gev0.6}  and \eqref{gev1.3DefK} and using the change of variables \eqref{gev8.0}, we introduce the new states
$$
\widetilde{w}_\eta \ov \widetilde b (w_\eta \circ \widetilde{X}), \quad \widetilde{q}_\eta\ov q_\eta\circ \widetilde{X},
$$
$$
\widetilde{\varphi}_\eta \ov \widetilde b(\varphi_\eta\circ \widetilde{X}), \quad \widetilde{\pi}_\eta \ov \pi_\eta\circ \widetilde{X}.
$$

From the above relations, systems \eqref{gev0.6}  and \eqref{gev1.3DefK} are transformed into the following systems
\begin{equation}\label{gev0.6Tilde}
\left\{\begin{array}{c}
\lambda \widetilde w_\eta -\nu\mathbb{L}\widetilde w_\eta + \mathbb{G} \widetilde{q}_\eta = 0 \quad \text{in} \ \mathcal{F}_{0},\\
\div  \widetilde w_\eta =   0  \quad \text{in} \ \mathcal{F}_{0},\\
\widetilde w_\eta = \Lambda_0 \eta  \quad \text{on} \  \partial \mathcal{F}_0
\end{array}\right.
\end{equation}
and
\begin{equation}\label{gev3.9Tilde}
\left\{\begin{array}{c}
\overline{\lambda} \widetilde \varphi_\eta 
-\nu \mathbb{L}\widetilde \varphi_\eta + \mathbb{G} \widetilde \pi_\eta
	= \widetilde w_\eta\quad \text{in} \ \mathcal{F}_{0},\\
\div  \widetilde \varphi_\eta =   0  \quad \text{in} \ \mathcal{F}_{0},\\
 \widetilde \varphi_\eta= 0  \quad \text{on} \  \partial \mathcal{F}_0.
\end{array}\right.
\end{equation}

Here, we have also transformed the operator $\Lambda$ defined by \eqref{ws9.1} into the operator 
$\Lambda_0 : L^2(\mathcal{I}) \to {\bf L}^2(\partial \mathcal{F}_0)$ defined by
\begin{equation}\label{ws9.1-0}
\begin{cases}
(\Lambda_0 \eta)(s,1) &= \left(M \eta(s)\right) e_2\\
 (\Lambda_0 \eta)(s,0)&=0
 \end{cases},
 \quad s\in \mathcal{I}.
\end{equation}

From \eqref{DefK} and \eqref{ws9.3}, we have the following formula
\begin{equation}\label{DefK-long}
K_{\lambda}\eta=-M\left( \nu D(\varphi_\eta)(s, 1+\eta_1^0(s))\begin{bmatrix}
-\partial_s \eta_1^0(s)
\\
1
\end{bmatrix}
\cdot e_2
-\pi_\eta(s, 1+\eta_1^0(s))
\right).
\end{equation}
Thus
\begin{equation}\label{DefK-tilde}
K_{\lambda}\eta=-M\left(\nu \mathbb{D}\widetilde{\varphi_\eta}(s, 1)-\widetilde\pi_\eta(s, 1)
\right),
\end{equation}
with
\begin{multline}\label{np0015}
\mathbb{D}\widetilde{\varphi_\eta}=
\frac 12 (-\partial_s \eta_1^0)\left(\sum_k \frac{\partial \widetilde a_{2k}}{\partial x_1}(\widetilde X) (\widetilde{\varphi_\eta})_k+ \sum_{k,\ell}  \widetilde a_{2k}(\widetilde X) \frac{\partial (\widetilde{\varphi_\eta})_k}{\partial y_\ell} 
\frac{\partial \widetilde Y_\ell}{\partial x_1}(\widetilde X)
+ \sum_{k,\ell}  \widetilde a_{1k}(\widetilde X) \frac{\partial (\widetilde{\varphi_\eta})_k}{\partial y_\ell} \frac{\partial \widetilde Y_\ell}{\partial x_2}(\widetilde X)\right)
\\
+\sum_k \frac{\partial \widetilde a_{2k}}{\partial x_2}(\widetilde X) (\widetilde{\varphi_\eta})_k+ \sum_{k,\ell}  \widetilde a_{2k}(\widetilde X) \frac{\partial (\widetilde{\varphi_\eta})_k}{\partial y_\ell} \frac{\partial \widetilde Y_\ell}{\partial x_2}(\widetilde X).
\end{multline}

In what follows, we write the above operators by splitting the derivatives with respect to $y_1$ and $y_2$. More precisely, we
introduce the set $\mathcal{O}_{\alpha}^{\beta}$ $(\beta\geq \alpha)$ of operators of the form
\begin{equation}\label{DefDkl-tt}
f \mapsto \sum_{i\leq \alpha} c_{i}^{(\beta-i)} \partial_1^i f,
\end{equation}
where $c_{i}^{(k)}$ are functions of the form
\begin{equation}\label{DefDkl-tt2}
c_{i}^{(k)}=\widehat{c_i}(\eta_1^0(y_1),\dots,\partial_{s}^{k}\eta_1^0(y_1),y_2 ),
\end{equation}
with $\widehat{c_i}$ a smooth function and $k\in \mathbb{N}$. These operators are thus depending on $y_2$ but it can be seen as a parameter.

For instance, using \eqref{bea4.5}--\eqref{bea5.1}, we deduce
\begin{equation}\label{np0014}
\mathbb{D} f=\mathbb{D}^{(1)} f + \mathbb{D}^{(0)} \partial_2 f, \quad
\mathbb{L} f=\mathbb{L}^{(2)} f + \mathbb{L}^{(1)} \partial_2 f+ \mathbb{L}^{(0)} \partial_{2}^2 f,\quad
\mathbb{G} f=\mathbb{G}^{(1)} f + \mathbb{G}^{(0)} \partial_2 f, 
\end{equation}
where $\mathbb{L}^{(2)} \in \mathcal{O}_2^3$, $\mathbb{D}^{(1)}, \mathbb{L}^{(1)}, \mathbb{G}^{(1)}\in \mathcal{O}_1^2$, $\mathbb{D}^{(0)}, \mathbb{L}^{(0)}, \mathbb{G}^{(0)} \in \mathcal{O}_0^1$.

\subsection{Commutator estimate}
First we show the following result:
\begin{Proposition}\label{Pcom}
Assume $\mathbb{B}\in \mathcal{O}_{\alpha}^\beta$ and $\eta_1^0\in W^{4+\beta,\infty}(\mathcal{I})$. For any $\theta\in (0,1)$ and for any 
$s\geq 0$, if $s>4\theta +\alpha-1$, then there exists $C>0$ such that
\begin{equation}\label{np0030}
\left\| [A_1^\theta M,\mathbb{B}] f\right\|_{L^2(\mathcal{I})} \leq C  \|f \|_{H^{s}(\mathcal{I})} \quad (f\in H^{s}(\mathcal{I})).
\end{equation}
\end{Proposition}
\begin{proof}
We write
$$
\mathbb{B}f=\sum_{i\leq \alpha} c_{i}^{(\beta-i)} \partial_1^i f.
$$
Then we recall the following formula (see for instance \cite[p. 98]{Triebel}):
$$
A_1^\theta=\frac{1}{\Gamma(\theta)\Gamma(1-\theta)} \int_0^\infty t^{\theta-1}A_1 (t I+A_1)^{-1} \ dt.
$$
In particular,
\begin{equation}\label{np0026}
[A_1^\theta M,\mathbb{B}]=A_1^\theta M\mathbb{B} - \mathbb{B}  A_1^\theta M
=A_1^\theta M\mathbb{B}M - M\mathbb{B}  A_1^\theta M-(I-M)\mathbb{B}  A_1^\theta M
+A_1^\theta M \mathbb{B}(I-M).
\end{equation}
By using the identity $A_1(t I+A_1)^{-1}=I-t (t I+A_1)^{-1}$ we deduce that
\begin{multline*}
A_1(t I+A_1)^{-1}M \mathbb{B}M-M \mathbb{B}A_1(t I+A_1)^{-1}M=-t (t I+A_1)^{-1} M \mathbb{B}M+t M \mathbb{B}(t I+A_1)^{-1}M\\
=t(t I+A_1)^{-1}(-M\mathbb{B}(tI+A_1)+(tI+A_1)M\mathbb{B}) (t I+A_1)^{-1}M=t(t I+A_1)^{-1}(-M\mathbb{B}A_1+A_1M\mathbb{B}) (t I+A_1)^{-1}M
\end{multline*}
and the first two terms in \eqref{np0026} give
$$
A_1^\theta M\mathbb{B}M - M\mathbb{B}  A_1^\theta M
=\frac{1}{\Gamma(\theta)\Gamma(1-\theta)} \int_0^\infty t^{\theta}  (t I+A_1)^{-1}M [A_1M,\mathbb{B}] (t I+A_1)^{-1}M\ dt.
$$
Moreover,
\begin{multline*}
[A_1M,\mathbb{B}] f
	=\left(\alpha_1 \partial_{1}^{4}  -\alpha_2 \partial_{1}^{2}\right) M \sum_{i\leq \alpha} c_{i}^{(\beta-i)} \partial_1^i f 
		-\sum_{i\leq \alpha} c_{i}^{(\beta-i)} \partial_1^i \left(\alpha_1 \partial_{1}^{4}  -\alpha_2 \partial_{1}^{2}\right) M f
	\\
	=\left(\alpha_1 \partial_{1}^{4}  -\alpha_2 \partial_{1}^{2}\right)  \sum_{i\leq \alpha} c_{i}^{(\beta-i)} \partial_1^i f 
		-\sum_{i\leq \alpha} c_{i}^{(\beta-i)} \partial_1^i \left(\alpha_1 \partial_{1}^{4}  -\alpha_2 \partial_{1}^{2}\right) f
\end{multline*}
and thus
\begin{equation}\label{EqMehCommutateur1}
[A_1M,\mathbb{B}]\in  \mathcal{O}_{\alpha+3}^{\beta+4}.
\end{equation}
We deduce that for $f\in H^{s}(\mathcal{I})$,
\begin{multline}\label{np0013}
\| (A_1^\theta M\mathbb{B}M - M\mathbb{B}  A_1^\theta M)f\|_{L^2(\mathcal{I})} 
\leq \frac{1}{\Gamma(\theta)\Gamma(1-\theta)} \int_0^\infty t^{\theta}  \|(t I+A_1)^{-1}M [A_1M,\mathbb{B}] (t I+A_1)^{-1}Mf \|_{L^2(\mathcal{I})}  \ dt
\\
\leq C \int_0^\infty t^{\theta-1}  \|  [A_1M,\mathbb{B}] (t I+A_1)^{-1}Mf \|_{L^2(\mathcal{I})}  \ dt
\leq C \int_0^\infty t^{\theta-1}  \| (t I+A_1)^{-1}Mf \|_{H^{\alpha+3}(\mathcal{I})} \ dt
\\
\leq C \left(\int_0^1 t^{\theta-1}  \| A_1(t I+A_1)^{-1}A_1^{\frac{\alpha-1}{4}}Mf \|_{L^2(\mathcal{I})} \ dt + \int_1^\infty t^{\theta-1} \| A_1^{\frac{\alpha+3-s}{4}}(t I+A_1)^{-1} A_1^{\frac{s}{4}}Mf \|_{L^2(\mathcal{I})} \ dt \right)\\
\leq C \left(\int_0^1 t^{\theta-1}  \| Mf \|_{H^{\alpha-1}(\mathcal{I})} \ dt + \int_1^\infty t^{\theta-1} \frac{1}{t^{1-(\alpha+3-s)/4}}  \| Mf \|_{H^{s}(\mathcal{I})} \ dt \right)
\leq C  \| M f \|_{H^{s}(\mathcal{I})}.
\end{multline}

For the third term in \eqref{np0026}, we write
\begin{multline}\label{np0028}
(I-M)\mathbb{B}  A_1^\theta Mf=\frac{1}{L} \int_0^L \left(\sum_{i\leq \alpha} c_{i}^{(\beta-i)} \partial_1^i A_1^\theta Mf\right) \ ds
= \frac{1}{L} \int_0^L  \left(\sum_{i\leq \alpha} (-1)^i  A_1^\theta M \partial_1^i c_{i}^{(\beta-i)}\right)  f \ ds\\
 \frac{1}{L} \int_0^L  A_1^{\theta-1}  \left(\sum_{i\leq \alpha} (-1)^i  M (\alpha_1\partial_1^{i+4}-\alpha_2 \partial_1^{i+2}) c_{i}^{(\beta-i)}\right)  f \ ds
\end{multline}
that yields
\begin{equation}\label{np0029}
\|(I-M)\mathbb{B}  A_1^\theta Mf\|_{L^2(\mathcal{I})}
\leq C\left(\sum_{i\leq \alpha} \|c_i^{(\beta-i)}\|_{H^{i+4}(\mathcal{I})}\right) \|f\|_{L^2(\mathcal{I})}.
\end{equation}

For the last term in \eqref{np0026}, we simply write
$$
\|(A_1^\theta M \mathbb{B})(I-M)f\|_{L^2(\mathcal{I})}
=\left\|(A_1^\theta M c_0^{(\beta)})\frac{1}{L} \int_0^L f \ ds \right\|_{L^2(\mathcal{I})} \leq C\|c_0^{(\beta)}\|_{H^{4\theta}(\mathcal{I})} \|f\|_{L^2(\mathcal{I})}.
$$

Combining this, \eqref{np0029} and \eqref{np0013}, we deduce the result.
\end{proof}

From \cref{Pcom} and \eqref{np0014}, we deduce in particular that if $\eta_1^0\in W^{7,\infty}(\mathcal{I})$ then for all $\varepsilon>0$ there exists $C=C(\varepsilon)>0$ such that 
\begin{equation}\label{np0032-0}
\left\|[A_1^{3/8} M, \mathbb{D}]f \right\|_{L^{2}(\mathcal{I})}\leq C \left\|f \right\|_{H^{3/2+\varepsilon}(\mathcal{I})}
\end{equation}
and
\begin{equation}\label{np0032}
\left\|[A_1^{3/8} M, \mathbb{L}]f \right\|_{{\bf L}^{2}(\mathcal{F}_0)}\leq C \left\|f \right\|_{{\bf H}^{5/2+\varepsilon}(\mathcal{F}_0)},
\quad
\left\| [A_1^{3/8} M, \mathbb{G}] f\right\|_{{\bf L}^{2}(\mathcal{F}_0)}
\leq C \left\|f \right\|_{{\bf H}^{3/2+\varepsilon}(\mathcal{F}_0)}.
\end{equation}

We are in position to prove \cref{2017Commutateur1}:
\begin{proof}[Proof of \cref{2017Commutateur1}]
We recall that $K_\lambda$ is given by \eqref{DefK}, \eqref{gev1.3DefK}, \eqref{bea0.1} and \eqref{gev0.6}.
After the change of variables, we have formulas \eqref{DefK-tilde}, \eqref{np0015} and \eqref{np0014} and thus
\begin{equation}\label{np0023}
[A_1^{3/8}, K_{\lambda}]\eta
=M\left[\nu \mathbb{D},A_1^{3/8}M\right]\widetilde{\varphi}_\eta(s, 1)
-M\nu \mathbb{D}\widetilde\varphi(s, 1)+M\widetilde \pi(s, 1)
\end{equation}
where
\begin{equation}\label{np0024}
\widetilde\varphi=A_1^{3/8} M \widetilde \varphi_{\eta}- \widetilde \varphi_{A_1^{3/8} \eta}\quad \mbox{ and }\quad \widetilde \pi= A_1^{3/8} M \widetilde \pi_{\eta}- \widetilde \pi_{A_1^{3/8} \eta}.
\end{equation}
From \eqref{gev0.6Tilde} and \eqref{gev3.9Tilde}, we have
\begin{equation}\label{gev3.9Tilde-diff}
\left\{\begin{array}{c}
\overline{\lambda} \widetilde \varphi
-\nu \mathbb{L}\widetilde \varphi + \mathbb{G} \widetilde \pi
	= \widetilde w
	+[A_1^{3/8} M, \nu \mathbb{L}]\widetilde \varphi_\eta - [A_1^{3/8} M, \mathbb{G}] \widetilde \pi_\eta
	\quad \text{in} \ \mathcal{F}_{0},\\
\div  \widetilde \varphi =   0  \quad \text{in} \ \mathcal{F}_{0},\\
 \widetilde \varphi= 0  \quad \text{on} \  \partial \mathcal{F}_0
\end{array}\right.
\end{equation}
where
$$
\widetilde w=A_1^{3/8} M \widetilde w_{\eta}- \widetilde w_{A_1^{3/8} \eta}\quad \mbox{ and }\quad \widetilde q= A_1^{3/8} M \widetilde{q}_{\eta}- \widetilde{q}_{A_1^{3/8} \eta}
$$
satisfy
\begin{equation}\label{gev0.6Tilde-diff}
\left\{\begin{array}{c}
\lambda \widetilde w -\nu \mathbb{L}\widetilde w + \mathbb{G} \widetilde{q} = 
	[A_1^{3/8} M, \nu \mathbb{L}]\widetilde w_\eta - [A_1^{3/8} M, \mathbb{G}] \widetilde q_\eta
 \quad \text{in} \ \mathcal{F}_{0},\\
\div  \widetilde w =   0  \quad \text{in} \ \mathcal{F}_{0},\\
\widetilde w = 0   \quad \text{on} \  \partial \mathcal{F}_0.
\end{array}\right.
\end{equation}
From \eqref{EstStokeslambdanh23novBis}, \eqref{EstStokeslambdanh23nov} and using the change of variables in \cref{toflat}, we deduce
\begin{equation}\label{EqMeh14092017-4}
\|\widetilde w\|_{\mathbf{L}^{2}(\mathcal{F}_0)}\leq C|\lambda|^{-1}
	\left\|[A_1^{3/8} M, \nu \mathbb{L}]\widetilde w_\eta - [A_1^{3/8} M, \mathbb{G}] \widetilde q_\eta \right\|_{{\bf L}^{2}(\mathcal{F}_0)}
\end{equation}
and
\begin{equation}\label{EqMeh14092017-5}
\|\widetilde \varphi\|_{\textbf{H}^{2}(\mathcal{F}_0)}+\|\widetilde \pi\|_{H^{1}(\mathcal{F}_0)}
\leq C\left\|\widetilde w
	+[A_1^{3/8} M, \nu \mathbb{L}]\widetilde \varphi_\eta - [A_1^{3/8} M, \mathbb{G}] \widetilde \pi_\eta
	\right\|_{{\bf L}^2(\mathcal{F}_0)}.
\end{equation}
From \eqref{EqMeh14092017-4}, \eqref{np0032},  \eqref{eq23-06-2017-2Bis} and using the change of variables in \cref{toflat}, we deduce
\begin{equation}\label{np0021}
\|\widetilde w\|_{\mathbf{L}^{2}(\mathcal{F}_0)}\leq C|\lambda|^{-1}
	\left(
	\left\|\widetilde w_\eta \right\|_{{\bf H}^{5/2+\varepsilon}(\mathcal{F}_0)} 
	+\left\|\widetilde q_\eta \right\|_{H^{3/2+\varepsilon}(\mathcal{F}_0)} 
	\right)
\leq C \left(|\lambda|^{-1} \|A_1^{1/2+\varepsilon/2}\eta\|_{\mathcal{H}_{\mathcal{S}}}+|\lambda|^{\varepsilon} \|\eta\|_{\mathcal{H}_{\mathcal{S}}}\right).
\end{equation}
Using \eqref{EqMeh14092017-5} and \eqref{np0032}, we find
\begin{equation}\label{np0020}
\|\widetilde \varphi\|_{\textbf{H}^{2}(\mathcal{F}_0)}+\|\widetilde \pi\|_{H^{1}(\mathcal{F}_0)}
\leq C\left(
	\|\widetilde w\|_{\textbf{L}^{2}(\mathcal{F}_0)}
	+\|\widetilde \varphi_\eta\|_{\textbf{H}^{5/2+\varepsilon}(\mathcal{F}_0)}
	+\|\widetilde \pi_{\eta}\|_{H^{3/2+\varepsilon}(\mathcal{F}_0)}
	\right).
\end{equation}
From \eqref{EstStokeslambdanh23nov}, \eqref{eq23-06-2017-1-0}, \eqref{eq23-06-2017-2} and using the change of variables in \cref{toflat}, we deduce
\begin{multline}\label{np0025}
\|\widetilde \varphi_\eta\|_{\textbf{H}^{5/2+\varepsilon}(\mathcal{F}_0)}
	+\|\widetilde \pi_{\eta}\|_{H^{3/2+\varepsilon}(\mathcal{F}_0)}
\leq C\left(|\lambda|^\varepsilon \|\widetilde w_\eta\|_{\textbf{H}^{1/2-\varepsilon}(\mathcal{F}_0)}+\|\widetilde w_\eta\|_{\textbf{H}^{1/2+\varepsilon}(\mathcal{F}_0)}\right)
\\
\leq C\left(|\lambda|^\varepsilon \|\eta\|_{\mathcal{H}_{\mathcal{S}}}+\|A_1^{\varepsilon/2} \eta\|_{\mathcal{H}_{\mathcal{S}}}\right).
\end{multline}
Combining the above equation with \eqref{np0021} and \eqref{np0020}, we deduce
\begin{equation}\label{np0022}
\|\widetilde \varphi\|_{\textbf{H}^{2}(\mathcal{F}_0)}+\|\widetilde \pi\|_{H^{1}(\mathcal{F}_0)}
\leq C \left(|\lambda|^{-1} \|A_1^{1/2+\varepsilon/2}\eta\|_{\mathcal{H}_{\mathcal{S}}}+|\lambda|^{\varepsilon} \|\eta\|_{\mathcal{H}_{\mathcal{S}}}+\|A_1^{\varepsilon/2} \eta\|_{\mathcal{H}_{\mathcal{S}}}\right).
\end{equation}

By using the above estimate, \eqref{np0023}, \eqref{np0025}, \eqref{np0032-0}, and trace estimates, 
\begin{multline*}
\left\| [A_1^{3/8}, K_{\lambda}]\eta \right\|_{\mathcal{H}_{\mathcal{S}}}
\leq \left\|  M\left[\mathbb{D},A_1^{3/8}M\right]\widetilde{\varphi}_\eta(\cdot, 1) \right\|_{\mathcal{H}_{\mathcal{S}}}
+
\left\| M\mathbb{D}\widetilde\varphi(\cdot, 1) \right\|_{\mathcal{H}_{\mathcal{S}}}
+
\left\| M\widetilde \pi(\cdot, 1) \right\|_{\mathcal{H}_{\mathcal{S}}}
\\
\leq C\left(
\|\widetilde \varphi_\eta\|_{\textbf{H}^{2+\varepsilon}(\mathcal{F}_0)}
+
\|\widetilde \varphi\|_{\textbf{H}^{2}(\mathcal{F}_0)}+\|\widetilde \pi\|_{H^{1}(\mathcal{F}_0)}
\right)
\leq C \left(|\lambda|^{-1} \|A_1^{1/2+\varepsilon/2}\eta\|_{\mathcal{H}_{\mathcal{S}}}+|\lambda|^{\varepsilon} \|\eta\|_{\mathcal{H}_{\mathcal{S}}}+\|A_1^{\varepsilon/2} \eta\|_{\mathcal{H}_{\mathcal{S}}}\right).
\end{multline*}
The conclusion follows from
$$
\|A_1^{\varepsilon/2} \eta\|_{\mathcal{H}_{\mathcal{S}}}\leq \left(|\lambda|^{\varepsilon} \|\eta\|_{\mathcal{H}_{\mathcal{S}}}\right)^{\frac{1}{1+\varepsilon}}\left(|\lambda|^{-1} \|A_1^{1/2+\varepsilon/2}\eta\|_{\mathcal{H}_{\mathcal{S}}}\right)^{\frac{\varepsilon}{1+\varepsilon}}\leq C\left(|\lambda|^{\varepsilon} \|\eta\|_{\mathcal{H}_{\mathcal{S}}}+|\lambda|^{-1} \|A_1^{1/2+\varepsilon/2}\eta\|_{\mathcal{H}_{\mathcal{S}}}\right).
$$
\end{proof}

\section{Estimation of $\widetilde{V}_{\lambda}^{-1}$} \label{sec_approx}
The aim of this section is to estimate $\widetilde{V}_{\lambda}^{-1}$ where $\widetilde{V}_{\lambda}$ is defined by \eqref{gev2.3}.
We recall that the notation $\mathbb{C}^+_\alpha$ is introduced in \eqref{gev2.4}.
\begin{Theorem}\label{T03}
There exists $\alpha>0$ such that for all $\lambda\in \mathbb{C}^+_\alpha$ the operator $\widetilde{V}_{\lambda} : \mathcal{D}(A_{1}) \to \mathcal{H}_{\mathcal{S}}$ is an isomorphism and for $\theta\in [0,1]$ the following estimates hold
\begin{equation}\label{EstV-CasComm}
\sup_{\lambda\in \mathbb{C}_{\alpha}^+}|\lambda|^{3/2-2\theta} \|A_1^\theta \widetilde{V}_{\lambda}^{-1}\|_{\mathcal{L}(\mathcal{H}_{\mathcal{S}})}<+\infty.
\end{equation} 
Moreover,
\begin{equation}\label{gev5.3}
\sup_{\lambda\in \mathbb{C}_{\alpha}^+}|\lambda|^{3/2-2\theta} \|A_1^\theta \widetilde{V}_{\lambda}^{*-1}\|_{\mathcal{L}(\mathcal{H}_{\mathcal{S}})}<+\infty.
\end{equation}
\end{Theorem}
\begin{proof}
Note that it is sufficient to consider the cases $\theta=0$ and $\theta=1$, the other cases are obtained by interpolation. 

Let us consider $\lambda\in \mathbb{C}_{\alpha}^+$ and $\eta\in \mathcal{D}(A_1)$. Then from \eqref{gev2.3}
\begin{equation}\label{np0000}
\left\|\frac{\widetilde{V}_{\lambda} \eta}{\lambda^2}\right\|_{\mathcal{H}_{\mathcal{S}}}^2
=\left\| (I+K_{\lambda})\eta+ \frac{2\rho A_1^{1/4}\eta}{\lambda}+ \frac{A_1\eta}{\lambda^2} \right\|_{\mathcal{H}_{\mathcal{S}}}^2.
\end{equation}
Using \eqref{eq28-11-1b} and \cref{P03}, we deduce that
\begin{equation}\label{np0001}
\|\eta\|_{\mathcal{H}_{\mathcal{S}}} \leq \|(I+K_{\lambda})^{1/2}\eta\|_{\mathcal{H}_{\mathcal{S}}}\leq C \|\eta\|_{\mathcal{H}_{\mathcal{S}}}.
\end{equation}
Thus, combining \eqref{np0000} and \eqref{np0001}, we deduce
\begin{multline}\label{np0002}
\left\|\frac{\widetilde{V}_{\lambda} \eta}{\lambda^2}\right\|_{\mathcal{H}_{\mathcal{S}}}^2
\geq 
\left\|(I+K_{\lambda})^{1/2}\eta + 2\rho \frac{(I+K_{\lambda})^{-1/2} A_1^{1/4}\eta}{\lambda}+ \frac{(I+K_{\lambda})^{-1/2}  A_1\eta}{\lambda^2} \right\|_{\mathcal{H}_{\mathcal{S}}}^2
\\
=\left\|(I+K_{\lambda})^{-1/2}\left(\eta + K_{\lambda}\eta  + \frac{A_1 \eta}{\lambda^2}\right) \right\|_{\mathcal{H}_{\mathcal{S}}}^2
	+4\rho^2\left\| \frac{(I+K_{\lambda})^{-1/2} A_1^{1/4}\eta}{\lambda}\right\|_{\mathcal{H}_{\mathcal{S}}}^2\\
	+4\rho \Re\left(\frac{1}{\lambda}\right) \left\|A_1^{1/8}\eta\right\|_{\mathcal{H}_{\mathcal{S}}}^2
	+4\rho \Re \left(\frac{1}{|\lambda|^2\lambda}\left\langle (I+K_{\lambda})^{-1} A_1\eta, A_1^{1/4}\eta\right\rangle_{\mathcal{H}_S}\right).
\end{multline}
We can write
\begin{multline}\label{np0003}
\left\langle (I+K_{\lambda})^{-1} A_1\eta, A_1^{1/4}\eta\right\rangle_{\mathcal{H}_S}
=
\|(I+K_{\lambda})^{-1/2} A_1^{5/8}\eta \|_{\mathcal{H}_S}^2
+\left\langle \left[(I+K_{\lambda})^{-1},A_1^{3/8}\right] A_1^{5/8}\eta, A_1^{1/4}\eta\right\rangle_{\mathcal{H}_S}
\\
=
\|(I+K_{\lambda})^{-1/2} A_1^{5/8}\eta \|_{\mathcal{H}_S}^2
+\left\langle  (I+K_{\lambda})^{-1}A_1^{5/8}\eta, \left[A_1^{3/8},K_{\lambda}\right](I+K_{\lambda})^{-1}A_1^{1/4}\eta\right\rangle_{\mathcal{H}_S}.
\end{multline}
Let use introduce the following notation,
\begin{equation}\label{DefClambda}
\mathcal{C}_\lambda\ov 4\rho \Re \left(\frac{1}{|\lambda|^2\lambda}\left\langle  (I+K_{\lambda})^{-1}A_1^{5/8}\eta, \left[A_1^{3/8}, K_{\lambda}\right](I+K_{\lambda})^{-1}A_1^{1/4}\eta\right\rangle_{\mathcal{H}_S}\right).
\end{equation}
Using \eqref{np0001} and \eqref{np0003} in \eqref{np0002} we deduce,
\begin{equation}\label{np0002imp}
\left\|\frac{\widetilde{V}_{\lambda} \eta}{\lambda^2}\right\|_{\mathcal{H}_{\mathcal{S}}}^2
\geq 
C\left(\left\|\eta + K_{\lambda}\eta  + \frac{A_1 \eta}{\lambda^2} \right\|_{\mathcal{H}_{\mathcal{S}}}^2
	+\left\| \frac{A_1^{1/4}\eta}{\lambda}\right\|_{\mathcal{H}_{\mathcal{S}}}^2\right)
	+\mathcal{C}_\lambda
\end{equation}
and thus
\begin{equation}\label{np0038bis}
\left\|\frac{\widetilde{V}_{\lambda} \eta}{\lambda^2}\right\|_{\mathcal{H}_{\mathcal{S}}}^2
\geq C\left(\|(I + K_{\lambda})\eta\|_{\mathcal{H}_{\mathcal{S}}}^2+\left\|\frac{A_1 \eta}{\lambda^2} \right\|_{\mathcal{H}_{\mathcal{S}}}^2-2\left|\left\langle \frac{A_1 \eta}{\lambda^2},\eta + K_{\lambda}\eta\right\rangle_{\mathcal{H}_{\mathcal{S}}}\right|
	+\left\| \frac{A_1^{1/4}\eta}{\lambda}\right\|_{\mathcal{H}_{\mathcal{S}}}^2\right)
	+\mathcal{C}_\lambda.
\end{equation}

Let us estimate 
$$
\left|\left\langle \frac{A_1 \eta}{\lambda^2},(I + K_{\lambda})\eta\right\rangle_{\mathcal{H}_{\mathcal{S}}}\right|.
$$
To do this, we start by noticing that 
$$
\left|\left\langle A_1 \eta,\eta\right\rangle_{\mathcal{H}_{\mathcal{S}}}\right|
\leq C\|A_1^{1/4} \eta\|_{\mathcal{H}_{\mathcal{S}}} \|A_1^{3/4} \eta\|_{\mathcal{H}_{\mathcal{S}}}.
$$
On the other hand, using \eqref{eq25112018-1},
$$
\left|\left\langle A_1 \eta,K_{\lambda}\eta\right\rangle_{\mathcal{H}_{\mathcal{S}}}\right|
\leq |\lambda|^\varepsilon \|\eta\|_{\mathcal{H}_{\mathcal{S}}} \|A_1^{3/4} \eta\|_{\mathcal{H}_{\mathcal{S}}},
$$
so that 
\begin{equation}\label{np0047}
\left|\left\langle A_1 \eta,(I + K_{\lambda})\eta\right\rangle_{\mathcal{H}_{\mathcal{S}}}\right| 
	\leq C \|A_1^{3/4} \eta\|_{\mathcal{H}_{\mathcal{S}}}\Bigg(\|A_1^{1/4} \eta\|_{\mathcal{H}_{\mathcal{S}}}
	+|\lambda|^{\varepsilon}\|\eta\|_{\mathcal{H}_{\mathcal{S}}}\Bigg).
\end{equation}
From interpolation inequality and Young inequality we deduce
\begin{equation}\label{np0045}
\frac{\left\| A_1^{3/4}\eta\right\|_{\mathcal{H}_S}}{|\lambda|^{3/2}}
\leq C\left(\frac{\|A_1\eta\|_{\mathcal{H}_{\mathcal{S}}}}{|\lambda|^{2}}\right)^{3/4}\|\eta\|_{\mathcal{H}_{\mathcal{S}}}^{1/4}
\leq C\left(\left\|\frac{A_1\eta}{\lambda^{2}}\right\|_{\mathcal{H}_{\mathcal{S}}}+\|\eta\|_{\mathcal{H}_{\mathcal{S}}}\right).
\end{equation}
Moreover, from the interpolation inequality $\|A_1^{3/4} \eta\|_{\mathcal{H}_{\mathcal{S}}}\leq C\|A_1^{1/4} \eta\|_{\mathcal{H}_{\mathcal{S}}}^{1/3}\|A_1 \eta\|_{\mathcal{H}_{\mathcal{S}}}^{2/3}$, we also deduce
\begin{multline}\label{np0048}
|\lambda|^{\varepsilon} \|A_1^{3/4} \eta\|_{\mathcal{H}_{\mathcal{S}}}\|\eta\|_{\mathcal{H}_{\mathcal{S}}}
\leq C |\lambda|^{\varepsilon} \|A_1^{1/4} \eta\|_{\mathcal{H}_{\mathcal{S}}}^{1/3} \|A_1 \eta\|_{\mathcal{H}_{\mathcal{S}}}^{2/3}\|\eta\|_{\mathcal{H}_{\mathcal{S}}}
\leq C |\lambda|^{4/3+\varepsilon} \|A_1^{1/4} \eta\|_{\mathcal{H}_{\mathcal{S}}} \left\|\frac{A_1 \eta}{\lambda^2}\right\|_{\mathcal{H}_{\mathcal{S}}}^{2/3}\|\eta\|_{\mathcal{H}_{\mathcal{S}}}^{1/3}\\
\leq C 
|\lambda|^{4/3+\varepsilon}\|A_1^{1/4} \eta\|_{\mathcal{H}_{\mathcal{S}}} \left(\left\|\frac{A_1\eta}{\lambda^{2}}\right\|_{\mathcal{H}_{\mathcal{S}}}+\|\eta\|_{\mathcal{H}_{\mathcal{S}}}\right).
\end{multline}
Hence, combining \eqref{np0047}, \eqref{np0045} and \eqref{np0048} with $\varepsilon<1/6$ yields
\begin{equation}\label{np0047bis}
\left|\left\langle A_1 \eta,(I + K_{\lambda})\eta\right\rangle_{\mathcal{H}_{\mathcal{S}}}\right| 
	\leq C |\lambda|^{3/2} \|A_1^{1/4} \eta\|_{\mathcal{H}_{\mathcal{S}}} \left(\left\|\frac{A_1\eta}{\lambda^{2}}\right\|_{\mathcal{H}_{\mathcal{S}}}+\|\eta\|_{\mathcal{H}_{\mathcal{S}}}\right),
\end{equation}
and we thus obtain 
$$
\left|\left\langle \frac{A_1 \eta}{\lambda^2},(I + K_{\lambda})\eta\right\rangle_{\mathcal{H}_{\mathcal{S}}}\right| 
	\leq C |\lambda| \left\|\frac{A_1^{1/4} \eta}{\lambda}\right\|_{\mathcal{H}_{\mathcal{S}}}^2
	+ \frac{1}{4} \left(\left\|\frac{A_1\eta}{\lambda^{2}}\right\|_{\mathcal{H}_{\mathcal{S}}}^2+\|(I + K_{\lambda})\eta\|_{\mathcal{H}_{\mathcal{S}}}^2 \right).
$$
Then with \eqref{np0038bis} it yields
$$
|\mathcal{C}_\lambda|+|\lambda|\left\| \frac{A_1^{1/4}\eta}{\lambda}\right\|_{\mathcal{H}_{\mathcal{S}}}^2+\left\|\frac{\widetilde{V}_{\lambda} \eta}{\lambda^2}\right\|_{\mathcal{H}_{\mathcal{S}}}^2
\geq C\left(\|(I + K_{\lambda})\eta\|_{\mathcal{H}_{\mathcal{S}}}^2+\left\|\frac{A_1 \eta}{\lambda^2} \right\|_{\mathcal{H}_{\mathcal{S}}}^2\right),
	$$
and with \eqref{np0002imp}, since $|\lambda|>\alpha$ with $\alpha>0$, we deduce
\begin{equation}\label{est062019}
|\lambda||\mathcal{C}_\lambda|+|\lambda|\left\|\frac{\widetilde{V}_{\lambda} \eta}{\lambda^2}\right\|_{\mathcal{H}_{\mathcal{S}}}^2\geq C\left(\|(I + K_{\lambda})\eta\|_{\mathcal{H}_{\mathcal{S}}}^2+\left\|\frac{A_1 \eta}{\lambda^2} \right\|_{\mathcal{H}_{\mathcal{S}}}^2\right),
\end{equation}
where $\mathcal{C}_\lambda$ is the product defined in \eqref{DefClambda}.

Next, let us now estimate $|\lambda||\mathcal{C}_\lambda|$. Using \cref{2017Commutateur1} and  \cref{P02}, we first deduce,
\begin{multline}\label{np0035}
\left\| \left[K_{\lambda},A_1^{3/8}\right](I+K_{\lambda})^{-1}A_1^{1/4}\eta \right\|_{\mathcal{H}_S}
\\
\leq C( |\lambda|^{-1}\|A_1^{1/2+\varepsilon}(I+K_{\lambda})^{-1}A_1^{1/4}\eta\|_{\mathcal{H}_{\mathcal{S}}}
+|\lambda|^{\varepsilon}\|(I+K_{\lambda})^{-1}A_1^{1/4}\eta\|_{\mathcal{H}_{\mathcal{S}}})
\\
\leq C( 
|\lambda|^{-1}\|A_1^{3/4+\varepsilon}\eta\|_{\mathcal{H}_{\mathcal{S}}}
+|\lambda|^{-1/2+3\varepsilon} \|A_1^{1/4} \eta\|_{\mathcal{H}_{\mathcal{S}}}
+|\lambda|^{\varepsilon}\|A_1^{1/4}\eta\|_{\mathcal{H}_{\mathcal{S}}}
).
\end{multline}
Thus, with \eqref{DefClambda} we obtain,
\begin{equation}\label{np0052}
|\lambda||\mathcal{C}_\lambda|
\leq C|\lambda|^{-2}\|A_1^{5/8}\eta\|_{\mathcal{H}_S} 
	\left(|\lambda|^{-1}\|A_1^{3/4+\varepsilon}\eta\|_{\mathcal{H}_{\mathcal{S}}}+|\lambda|^{\varepsilon}\|A_1^{1/4}\eta\|_{\mathcal{H}_{\mathcal{S}}}\right).
\end{equation}
From interpolation inequality and Young inequality we deduce
$$
\frac{\left\| A_1^{5/8}\eta\right\|_{\mathcal{H}_S}}{|\lambda|^{5/4}}
\leq C\left(\frac{\|A_1\eta\|_{\mathcal{H}_{\mathcal{S}}}}{|\lambda|^{2}}\right)^{5/8}\|\eta\|_{\mathcal{H}_{\mathcal{S}}}^{3/8}
\leq C\left(\left\|\frac{A_1\eta}{\lambda^{2}}\right\|_{\mathcal{H}_{\mathcal{S}}}+\|\eta\|_{\mathcal{H}_{\mathcal{S}}}\right),
$$
$$
\frac{\left\| A_1^{3/4+\varepsilon}\eta\right\|_{\mathcal{H}_S}}{|\lambda|^{3/2+2\varepsilon}}
\leq C\left(\frac{\|A_1\eta\|_{\mathcal{H}_{\mathcal{S}}}}{|\lambda|^{2}}\right)^{3/4+\varepsilon}\|\eta\|_{\mathcal{H}_{\mathcal{S}}}^{1/4-\varepsilon}
\leq C\left(\left\|\frac{A_1\eta}{\lambda^{2}}\right\|_{\mathcal{H}_{\mathcal{S}}}+\|\eta\|_{\mathcal{H}_{\mathcal{S}}}\right),
$$
and
$$
\frac{\left\| A_1^{1/4}\eta\right\|_{\mathcal{H}_S}}{|\lambda|^{1/2}}
\leq C\left(\frac{\|A_1\eta\|_{\mathcal{H}_{\mathcal{S}}}}{|\lambda|^{2}}\right)^{1/4}\|\eta\|_{\mathcal{H}_{\mathcal{S}}}^{3/4}
\leq C\left(\left\|\frac{A_1\eta}{\lambda^{2}}\right\|_{\mathcal{H}_{\mathcal{S}}}+\|\eta\|_{\mathcal{H}_{\mathcal{S}}}\right).
$$
The above estimates with \eqref{np0052} yield,
\begin{equation}\label{np0054}
|\lambda||\mathcal{C}_\lambda|
 \leq C|\lambda|^{-1/4+2\varepsilon}\left(\left\|\frac{A_1\eta}{\lambda^{2}}\right\|_{\mathcal{H}_{\mathcal{S}}}^2+\|\eta\|_{\mathcal{H}_{\mathcal{S}}}^2\right),
\end{equation}
and by combining \eqref{est062019} and \eqref{np0054}, for $|\lambda|>\alpha$ and $\alpha>0$ large enough we deduce,
\begin{equation}\label{gev5.20}
 	\left\|\widetilde{V}_{\lambda} \eta\right\|_{\mathcal{H}_{\mathcal{S}}}^2
\geq C\left(|\lambda|^3 \|\eta\|_{\mathcal{H}_{\mathcal{S}}}^2+|\lambda|^{-1} \left\|A_1 \eta \right\|_{\mathcal{H}_{\mathcal{S}}}^2\right).
\end{equation}
Then \eqref{EstV-CasComm} is proved for $\theta=0$ and $\theta=1$ if we show that $\widetilde{V}_{\lambda}$ is invertible.

For that, we first deduce from \cref{P04} that
\begin{equation}\label{gev5.2}
\widetilde{V}_{\lambda}^*= \overline{\lambda}^2 (I+K_{\lambda}) + 2\rho \overline{\lambda} A_1^{1/4}+ A_1.
\end{equation}
Since $\overline{\lambda}\in \mathbb{C}_\alpha^+$ if ${\lambda}\in \mathbb{C}_\alpha^+$, we perform the same calculations as above that led to \eqref{gev5.20} to find 
\begin{equation}\label{gev5.3bis}
\left\|\widetilde{V}_{\lambda}^* \eta\right\|_{\mathcal{H}_{\mathcal{S}}}^2\geq  C \left(|\lambda|^{3}\|\eta\|_{\mathcal{H}_{\mathcal{S}}}^2+|\lambda|^{-1}\left\|A_1 \eta \right\|_{\mathcal{H}_{\mathcal{S}}}^2\right).
\end{equation}
Relation \eqref{gev5.3bis} yields that the image of $\widetilde{V}_{\lambda}$ is dense in $\mathcal{H}_{\mathcal{S}}$ and relation \eqref{gev5.20} implies that the image of $\widetilde{V}_{\lambda}$ is closed in $\mathcal{H}_{\mathcal{S}}$ and that $\widetilde{V}_{\lambda}$ is injective.  We then deduce that $\widetilde{V}_{\lambda}$ is invertible. 

Finally, \eqref{gev5.3bis} gives \eqref{gev5.3}. 

\end{proof}

\begin{Corollary}\label{Cor05072017}
Let $\alpha>0$ be given in Theorem \ref{T03}. For $\theta\in [0,1]$ and $\beta\in [0,1]$ such that $\theta+\beta\leq 1$ the following estimate holds
\begin{equation}\label{EstV-CasComm2}
\sup_{\lambda\in \mathbb{C}^+_{\alpha}}|\lambda|^{3/2-2\theta-2\beta} \|A_1^{\theta} \widetilde{V}_{\lambda}^{-1}A_1^{\beta}\|_{\mathcal{L}(\mathcal{H}_{\mathcal{S}})}<+\infty.
\end{equation}
\end{Corollary}
\begin{proof}
From \eqref{gev5.3} we deduce 
$$\sup_{\lambda\in \mathbb{C}_{\alpha}^+}|\lambda|^{-1/2}\|(\widetilde{V}_{\lambda}^{-1})^*\|_{\mathcal{L}(\mathcal{H}_{\mathcal{S}},\mathcal{D}(A_1))}<+\infty,$$ from which we obtain by duality,
\begin{equation}\label{np0057}
\sup_{\lambda\in \mathbb{C}_{\alpha}^+}|\lambda|^{-1/2} \|\widetilde{V}_{\lambda}^{-1}\|_{\mathcal{L}(\mathcal{D}(A_1)',\mathcal{H}_{\mathcal{S}})}<+\infty.
\end{equation}
By combining \eqref{np0057} with  \eqref{EstV-CasComm} for $\theta=1$ with an interpolation argument yields,
\begin{equation}\label{eq04-07-2017-1}
\sup_{\lambda\in \mathbb{C}_{\alpha}^+}|\lambda|^{-1/2} \|\widetilde{V}_{\lambda}^{-1}\|_{\mathcal{L}(\mathcal{D}(A_1^\beta)',\mathcal{D}(A_1^{1-\beta}))}<+\infty\quad (\beta\in [0,1]).
\end{equation}
By combining \eqref{np0057} with   \eqref{EstV-CasComm} for $\theta=0$ with an interpolation argument yields for $\beta\in [0,1]$,
\begin{equation}\label{eq04-07-2017-2}
\sup_{\lambda\in \mathbb{C}_{\alpha}^+}|\lambda|^{3/2-2\beta} \|\widetilde{V}_{\lambda}^{-1}\|_{\mathcal{L}(\mathcal{D}(A_1^\beta)',\mathcal{H}_{\mathcal{S}})}<+\infty\quad (\beta\in [0,1]).
\end{equation}
Then \eqref{EstV-CasComm2} follows by interpolating \eqref{eq04-07-2017-1} and \eqref{eq04-07-2017-2}.
\end{proof}

\begin{Proposition}\label{C01}
Let $\alpha>0$ be given in Theorem \ref{T03}. For $\theta\in [-1/8,1]$ the following estimates hold
\begin{equation}\label{EstV-22}
\sup_{\lambda\in \mathbb{C}^+_{\alpha}}|\lambda|^{7/4-2\theta} \|A_1^{-1/8} \widetilde{V}_{\lambda}^{-1}A_1^\theta\|_{\mathcal{L}(\mathcal{H}_{\mathcal{S}})}<+\infty,
\end{equation} 
\begin{equation}\label{EstV-22Adj}
\sup_{\lambda\in \mathbb{C}^+_{\alpha}}|\lambda|^{7/4-2\theta} \|A_1^{\theta} \widetilde{V}_{\lambda}^{-1}A_1^{-1/8}\|_{\mathcal{L}(\mathcal{H}_{\mathcal{S}})}<+\infty.
\end{equation} 
\end{Proposition}
\begin{proof}
In a first step, we prove the case $\theta=0$. For that we first observe that \eqref{eq25112018-2} with $\theta=-1/4$ yields:
\begin{equation}\label{eq-14042019}
\|A_1^{-1/8}\widetilde{V}_{\lambda}^{-1}\|_{\mathcal{L}(\mathcal{H}_{\mathcal{S}})}\leq C\|A_1^{-1/8}(I+K_\lambda)\widetilde{V}_{\lambda}^{-1}\|_{\mathcal{L}(\mathcal{H}_{\mathcal{S}})}.
\end{equation}
Next, we make the following calculations:
\begin{equation}\label{Eq-04-07-2017}
\lambda^2A_1^{-1/8} (I+K_\lambda) \widetilde{V}_{\lambda}^{-1}=A_1^{-1/8}-(2\rho \lambda A_1^{1/8}+A_1^{7/8})\widetilde{V}_{\lambda}^{-1}
\end{equation}
which yields with \eqref{EstV-CasComm} and \eqref{eq-14042019},
\begin{multline*}
|\lambda|^2 \|A_1^{-1/8}\widetilde{V}_{\lambda}^{-1}\|_{\mathcal{L}(\mathcal{H}_{\mathcal{S}})}\leq |\lambda|^2 \|A_1^{-1/8}  (I+K_\lambda) \widetilde{V}_{\lambda}^{-1}\|_{\mathcal{L}(\mathcal{H}_{\mathcal{S}})}\\
\leq  C \left(\|A_1^{-1/8}\|_{\mathcal{L}(\mathcal{H}_{\mathcal{S}})}+|\lambda|\|A_1^{1/8} \widetilde{V}_{\lambda}^{-1}\|_{\mathcal{L}(\mathcal{H}_{\mathcal{S}})}+\|A_1^{7/8}\widetilde{V}_{\lambda}^{-1}\|_{\mathcal{L}(\mathcal{H}_{\mathcal{S}})}\right)\leq C(1+|\lambda|^{-1/4}+|\lambda|^{1/4}).
\end{multline*}
This leads to \eqref{EstV-22} for $\theta=0$. Moreover, extimate \eqref{EstV-22} for $\theta=0$ but for $\widetilde{V}_{\lambda}^*$ instead of  $\widetilde{V}_{\lambda}$ follows analogously from \eqref{gev5.3}, and then \eqref{EstV-22Adj} for $\theta=0$ by a duality argument.

In a second step, let us prove the case $\theta=1$. For that we first observe,
$$
A_1 \widetilde{V}_{\lambda}^{-1}A_1^{-1/8}= A_1^{-1/8}-\lambda^2(I+K_\lambda)\widetilde{V}_{\lambda}^{-1}A_1^{-1/8}-2\rho \lambda A_1^{1/4} \widetilde{V}_{\lambda}^{-1}A_1^{-1/8}.
$$
Using \eqref{EstV-22Adj} for $\theta=0$ and \eqref{EstV-CasComm} for $\theta=1/4$ we obtain the estimate
 \begin{multline*}
 \|A_1 \widetilde{V}_{\lambda}^{-1}A_1^{-1/8}\|_{\mathcal{L}(\mathcal{H}_{\mathcal{S}})}\leq \| A_1^{-1/8}\|_{\mathcal{L}(\mathcal{H}_{\mathcal{S}})}+C\left(|\lambda|^2 \| \widetilde{V}_{\lambda}^{-1}A_1^{-1/8}\|_{\mathcal{L}(\mathcal{H}_{\mathcal{S}})}+|\lambda| \|A_1^{1/4} \widetilde{V}_{\lambda}^{-1}A_1^{-1/8}\|_{\mathcal{L}(\mathcal{H}_{\mathcal{S}})}\right)\\
 \leq C(1+|\lambda|^{1/4})\leq C|\lambda|^{1/4},
 \end{multline*}
where in the last inequality we have used the fact that $\lambda\in \mathbb{C}_\alpha^+$ with $\alpha>0$. Then \eqref{EstV-22Adj} for $\theta=1$ is proved. Finally, \eqref{EstV-22} for $\theta=1$ follows by duality, and \eqref{EstV-22}, \eqref{EstV-22Adj} for $\theta\in (0,1)$ follow by interpolation.
 
It remains to prove the case $\theta=-1/8$. The case $\theta\in (-1/8,0)$ will then follow by interpolation.  We come back to \eqref{Eq-04-07-2017} from which we deduce
\begin{equation}
\lambda^2A_1^{-1/8} (I+K_\lambda) \widetilde{V}_{\lambda}^{-1}A_1^{-1/8}
=A_1^{-1/4}-(2\rho \lambda A_1^{1/8}+A_1^{7/8})\widetilde{V}_{\lambda}^{-1}A_1^{-1/8},
\end{equation}
and with \eqref{eq-14042019}, \eqref{EstV-22} with $\theta=1/8$ and $\theta=7/8$, we get
\begin{multline*}
|\lambda|^2 \|A_1^{-1/8}\widetilde{V}_{\lambda}^{-1}A_1^{-1/8}\|_{\mathcal{L}(\mathcal{H}_{\mathcal{S}})}\leq |\lambda|^2 \|A_1^{-1/8}(I+K_\lambda) \widetilde{V}_{\lambda}^{-1}A_1^{-1/8}\|_{\mathcal{L}(\mathcal{H}_{\mathcal{S}})}\\
\leq  C \left(\|A_1^{-1/4}\|_{\mathcal{L}(\mathcal{H}_{\mathcal{S}})}+|\lambda|\|A_1^{1/8} \widetilde{V}_{\lambda}^{-1}A_1^{-1/8}\|_{\mathcal{L}(\mathcal{H}_{\mathcal{S}})}+\|A_1^{7/8}\widetilde{V}_{\lambda}^{-1}A_1^{-1/8}\|_{\mathcal{L}(\mathcal{H}_{\mathcal{S}})}\right)\leq C(1+|\lambda|^{-1/2})\leq C.
\end{multline*}
Then the case $\theta=-1/8$ is proved.
\end{proof}

\begin{Corollary}\label{C01Final}
Let $\alpha>0$ be given in Theorem \ref{T03}. For $\theta\in [-1/8,1]$ and $\beta\in [-1/8,1]$ such that $\theta+\beta\leq 1$ the following estimate holds
\begin{equation}\label{EstFinale}
\sup_{\lambda\in \mathbb{C}^+_{\alpha}}|\lambda|^{3/2-2\theta-2\beta} \|A_1^\theta \widetilde{V}_{\lambda}^{-1}A_1^\beta\|_{\mathcal{L}(\mathcal{H}_{\mathcal{S}})}<+\infty.
\end{equation} 
\end{Corollary}
\begin{proof}
First, assume $\theta\in [0,1]$. From \eqref{EstV-CasComm2} with $\beta=0$ and \eqref{EstV-22Adj} with an interpolation argument 
we deduce \eqref{EstFinale} for $\theta\in [0,1]$ and $\beta\in [-1/8,0]$. Then with \eqref{EstV-CasComm2} we deduce 
\eqref{EstFinale} for $\theta\in [0,1]$ and $\beta\in [-1/8,1]$. 

Finally, interpolating \eqref{EstFinale} with $\theta=0$ and $\beta\in [-1/8,1]$ and \eqref{EstV-22} with $\theta=\beta$ allows us to obtain \eqref{EstFinale} for $\theta\in[-1/8,0]$ and $\beta\in [-1/8,1]$.

\end{proof}
\section{Proof of Theorem \ref{CorMain}}\label{sec_res}
The goal of this section it to prove the Gevrey type resolvent estimates for the operator $A_0$ defined by \eqref{fs3.2}--\eqref{fs3.3}.
In order to prove Theorem \ref{CorMain}, we rewrite the resolvent equation in a more convenient way.
Assume $\lambda\in \mathbb{C}_\alpha^{+}$ for $\alpha>0$ given in Theorem \ref{T03} and $[f,g,h]\in \mathcal{H}$. We set 
$[v, \eta_1,\eta_2]\ov (\lambda-A_0)^{-1} [f,g,h]$ so that
\begin{equation}\label{gev0.4}
\left\{\begin{array}{c}
\lambda v -\div \mathbb{T}(v,p) = f \quad \text{in} \ \mathcal{F},\\
\div v =   0   \quad \text{in} \  \mathcal{F},\\
v = \Lambda \eta_2 \quad \text{on} \  \partial \mathcal{F},\\
\lambda \eta_1 -\eta_2 = g\\
\lambda \eta_2 + A_1 \eta_1
  =- \Lambda^*\left\{\mathbb{T}(v,p)n_{|\partial\mathcal{F}}\right\}+h.
\end{array}\right.
\end{equation}

Using $W_{\lambda}$ and $\mathbb{A}$ introduced in \eqref{bea0.1} and \eqref{gev9.4}, we can decompose the fluid velocity of \eqref{gev0.4} as 
$$
v= W_{\lambda}\eta_2+(\lambda I-\mathbb{A})^{-1}\mathbb{P}f,
$$
and using $L_{\lambda}\in \mathcal{L}(\mathcal{D}(A_1^{3/8}),\mathcal{D}(A_1^{1/8}))$ and 
$\mathcal{T}_{\lambda}\in \mathcal{L}({\bf L}^2(\mathcal{F}),\mathcal{D}(A_1^{1/8}))$ defined by \eqref{DefG} and \eqref{np0056} we can rewrite
 system \eqref{gev0.4} as
\begin{equation}\label{gev0.8}
\left\{\begin{array}{c}
\lambda \eta_1 -\eta_2 = g\\
\lambda \eta_2 + A_1 \eta_1+L_{\lambda}\eta_2
  =\mathcal{T}_{\lambda}f+h.
\end{array}\right.
\end{equation}
This writes
\begin{equation}\label{gev1.1}
\left(\lambda I +\mathcal{A}_\lambda\right) \begin{bmatrix}
\eta_1 \\ \eta_2
\end{bmatrix}
=\begin{bmatrix}
g \\ \mathcal{T}_{\lambda}f+h
\end{bmatrix}
\end{equation}
with
\begin{equation}\label{gev1.0}
\mathcal{A}_\lambda \ov
\begin{bmatrix}
0 & -I \\
A_1 & L_{\lambda}
\end{bmatrix}.
\end{equation}
We recall that $V_{\lambda}$ defined by \eqref{gev9.8} is invertible. It is a consequence of the following result proved in Proposition 4.8 of \cite{plat}.
\begin{Proposition}\label{Vinv}
For all $\lambda\in \mathbb{C}^+$ the operator $V_{\lambda}$ is an isomorphism from $\mathcal{D}(A_1)$ onto $\mathcal{H}_S$. 
\end{Proposition}
Hence, direct calculations lead to the following formulas for the inverse 
of $\lambda I +\mathcal{A}_{\lambda}$ and of $\lambda I-A_0$:
\begin{equation}\label{gev1.1BisBis}
\left(\lambda I +\mathcal{A}_{\lambda}\right)^{-1}=\begin{bmatrix}
\frac{I-V_{\lambda}^{-1}A_1}{\lambda} &V_{\lambda}^{-1} \\[2mm]
-V_{\lambda}^{-1}A_1 & \lambda V_{\lambda}^{-1}
\end{bmatrix},
\end{equation}
and 
\begin{equation}\label{EqresolventA0}
(\lambda I -A_0)^{-1}=  \displaystyle \begin{bmatrix}
(\lambda I-\mathbb{A})^{-1}\mathbb{P}+\lambda  W_{\lambda} V_{\lambda}^{-1}\mathcal{T}_{\lambda}& -W_{\lambda}V_{\lambda}^{-1}A_1 & \lambda W_{\lambda} V_{\lambda}^{-1}\\[2mm]
 V_{\lambda}^{-1} \mathcal{T}_{\lambda}& \frac{I-V_{\lambda}^{-1}A_1}{\lambda} &V_{\lambda}^{-1} \\[2mm]
\lambda V_{\lambda}^{-1}\mathcal{T}_{\lambda}  &-V_{\lambda}^{-1}A_1 & \lambda V_{\lambda}^{-1}
\end{bmatrix}.
\end{equation}

\subsection{Estimation of $V_{\lambda}^{-1}$} 
In this section, we estimate the inverse of the operator $V_{\lambda}$ defined in \eqref{gev9.8} for $\lambda\in \mathbb{C}_\alpha^+$ and $\alpha>0$ given in Theorem \ref{T03}. 
From now on, we fix $\rho<\rho_1/4$ where $\rho_{1}$ is defined in \cref{P04}.
The main result of this section is the following:
\begin{Theorem}\label{ThmEstV} Let $\alpha>0$ be given in Theorem \ref{T03}. For $\theta\in [-1/8,7/8]$ and $\beta\in [-1/8,7/8]$ such that $\theta+\beta\leq 1$ the following estimate holds
\begin{equation}\label{EstV-2}
\sup_{\lambda\in \mathbb{C}^+_{\alpha}}|\lambda|^{3/2-2\theta-2\beta} \|A_1^{\theta} V_{\lambda}^{-1}A_1^{\beta}\|_{\mathcal{L}(\mathcal{H}_{\mathcal{S}})}<+\infty.
\end{equation} 
\end{Theorem}
\begin{proof}
Comparing \eqref{gev9.8} and \eqref{gev2.3}, we see that
$$
V_{\lambda}-\widetilde{V}_{\lambda}=\lambda  S_{\lambda}, \quad S_{\lambda}\ov G_{\lambda}-2\rho A_1^{1/4}
$$
and thus
\begin{equation}\label{eq01032017}
[I+\lambda \widetilde{V}_{\lambda}^{-1}  S_{\lambda}]V_{\lambda}^{-1}=\widetilde{V}_{\lambda}^{-1}.
\end{equation}
We thus need to estimate the inverse of $[I+\lambda \widetilde{V}_{\lambda}^{-1}  S_{\lambda}]$. 

From \cref{P04} and in particular \eqref{eq28-11-1a}, we have $S_{\lambda}$ is a positive self-adjoint operator satisfying
$$
\|S^{1/2}_{\lambda}\eta\|_{\mathcal{H}_{\mathcal{S}}}\leq C\left(\|A_1^{1/8}\eta\|_{\mathcal{H}_{\mathcal{S}}}+|\lambda|^{1/2}\|A_1^{-1/8}\eta\|_{\mathcal{H}_{\mathcal{S}}}\right)
\quad (\eta\in \mathcal{D}(A_1^{1/8})).
$$
Combining the above inequality with \eqref{EstV-CasComm2} and \eqref{EstV-22} we obtain for $\beta\in [-1/8,7/8]$,
\begin{equation}\label{eq01032017-3}
\|S^{1/2}_{\lambda}\widetilde{V}_{\lambda}^{-1}A_1^{\beta}\eta\|_{\mathcal{H}_{\mathcal{S}}}\leq C |\lambda|^{-5/4+2\beta}\|\eta\|_{\mathcal{H}_{\mathcal{S}}}.
\end{equation}
Analogously we can prove \eqref{eq01032017-3} but for $\widetilde{V}_{\lambda}^*$ instead of $\widetilde{V}_{\lambda}$. Then a duality argument yield for $\theta\in [-1/8,7/8]$
\begin{equation}\label{eq01032017-4}
\|A_1^\theta \widetilde{V}_{\lambda}^{-1}S^{1/2}_{\lambda}\eta\|_{\mathcal{H}_{\mathcal{S}}}\leq C |\lambda|^{-5/4+2\theta}\|\eta\|_{\mathcal{H}_{\mathcal{S}}}.
\end{equation}

From \eqref{gev2.3}, we obtain that for any $\lambda \in \mathbb{C}^+$ and for any $\zeta \in \mathcal{D}(A_1)$,
\begin{equation}\label{gev5.6}
\Re \langle  \widetilde{V}_{\lambda}  \zeta, \lambda \zeta \rangle_{\mathcal{H}_{\mathcal{S}}}
=\Re\lambda \|\lambda (I+K_{\lambda})^{1/2}\zeta\|_{\mathcal{H}_{\mathcal{S}}}^2
+2\rho\| \lambda A_1^{1/8}\zeta\|_{\mathcal{H}_{\mathcal{S}}}^2
+\Re\lambda\|A_1^{1/2}\zeta\|_{\mathcal{H}_{\mathcal{S}}}^2 \geq 0.
\end{equation}
In particular, for any $\lambda \in \mathbb{C}^+$ and for any $\zeta \in \mathcal{H}_{\mathcal{S}}$,
\begin{equation}\label{gev5.7}
\Re \langle \lambda \widetilde{V}_{\lambda}^{-1}  \zeta, \zeta\rangle_{\mathcal{H}_{\mathcal{S}}} \geq 0.
\end{equation}

Let us now consider the equation
\begin{equation}\label{eq01032017-5}
\eta+\lambda \widetilde{V}_{\lambda}^{-1}  S_{\lambda}\eta=f.
\end{equation}
If we multiply \eqref{eq01032017-5} by $S_{\lambda}\eta$, take the real part and use \eqref{gev5.7}, we obtain 
$$
\|S_{\lambda}^{1/2}\eta\|_{\mathcal{H}_{\mathcal{S}}}\leq \|S_{\lambda}^{1/2}f\|_{\mathcal{H}_{\mathcal{S}}}
$$ and applying such a result to  equality \eqref{eq01032017} yields
\begin{equation}\nonumber 
\forall \eta\in \mathcal{H}_{\mathcal{S}},\quad  \|S^{1/2}_{\lambda}V_{\lambda}^{-1}\eta \|_{\mathcal{H}_{\mathcal{S}}}\leq \|S^{1/2}_{\lambda}\widetilde{V}_{\lambda}^{-1}\eta \|_{\mathcal{H}_{\mathcal{S}}}.
\end{equation}
Thus, coming back to equality \eqref{eq01032017} we deduce that for $\eta\in \mathcal{H}_{\mathcal{S}}$, $\theta\in [-1/8,7/8]$, $\beta\in [-1/8,7/8]$, $\theta+\beta\leq 1$,
\begin{multline*}
\|A_1^\theta V_{\lambda}^{-1}A_1^\beta \eta \|_{\mathcal{H}_{\mathcal{S}}}\leq \|A_1^\theta \widetilde{V}_{\lambda}^{-1}A_1^\beta \eta \|_{\mathcal{H}_{\mathcal{S}}}
	+|\lambda|\|A_1^\theta \widetilde{V}_{\lambda}^{-1}S_{\lambda}V_{\lambda}^{-1}A_1^\beta \eta\|_{\mathcal{H}_{\mathcal{S}}}\\
\leq \|A_1^\theta \widetilde{V}_{\lambda}^{-1}A_1^\beta \eta \|_{\mathcal{H}_{\mathcal{S}}}+|\lambda|\|A_1^\theta \widetilde{V}_{\lambda}^{-1}S^{1/2}_{\lambda}\|_{\mathcal{L}(\mathcal{H}_{\mathcal{S}})}\|S^{1/2}_{\lambda}V_{\lambda}^{-1}A_1^\beta\eta\|_{\mathcal{H}_{\mathcal{S}}}\\
\leq \|A_1^\theta \widetilde{V}_{\lambda}^{-1}A_1^\beta \eta \|_{\mathcal{H}_{\mathcal{S}}}+|\lambda|\|A_1^\theta \widetilde{V}_{\lambda}^{-1}S^{1/2}_{\lambda}\|_{\mathcal{L}(\mathcal{H}_{\mathcal{S}})}\|S^{1/2}_{\lambda}\widetilde{V}_{\lambda}^{-1}A_1^\beta \eta\|_{\mathcal{H}_{\mathcal{S}}}
\end{multline*}
Then using estimates \eqref{EstFinale}, \eqref{eq01032017-3} and \eqref{eq01032017-4} yields \eqref{EstV-2}.
\end{proof}
\subsection{Proof of Theorem \ref{CorMain}} 
\begin{proof}
First, the exponential stability of $(e^{A_0 t})_{t\geq 0}$ (see Proposition \ref{AStrongContSG}) and standard results (see \cite[p.101, Theorem 2.5]{RCIDS2ED}) yield 
that $\|(\lambda-A_0)^{-1}\|_{\mathcal{L}(\mathcal{H})}$ is uniformly bounded for $\lambda\in \mathbb{C}^+$. This implies that
$$
\sup_{\lambda\in \mathbb{C}^+,|\lambda|\leq \alpha} \left\{\|A_0(\lambda-A_0)^{-1}\|_{\mathcal{L}(\mathcal{H})}+ |\lambda| \|(\lambda-A_0)^{-1}\|_{\mathcal{L}(\mathcal{H})}\right\} <\infty.
$$
Using \eqref{gev7.4} and \eqref{fs3.4}, we deduce \eqref{ResEstA0-1} and \eqref{ResEstA0-3} for $\lambda \in \mathbb{C}^+$ with $|\lambda|\leq \alpha$.
In the remaining part of the proof, we can thus assume $\lambda\in \mathbb{C}_\alpha^+$ (see \eqref{gev2.4}) for $\alpha>0$ given in Theorem \ref{T03}. 
From \eqref{EstStokeslambdanh23novBis}, we have
\begin{equation}\label{11a}
|\lambda|\|(\lambda-\mathbb{A})^{-1}\mathbb P f\|_{{\bf L}^2(\mathcal{F})} \leq C \|f\|_{{\bf L}^2(\mathcal{F})}.
\end{equation}
From \eqref{eq23-06-2017-1-0},  \eqref{EstV-2} with $(\theta,\beta)=(0,0)$ and \eqref{np0056},
\begin{equation}\label{11b}
|\lambda|^{1/2} \|\lambda  W_{\lambda} V_{\lambda}^{-1}\mathcal{T}_{\lambda} f\|_{{\bf L}^2(\mathcal{F})} 
\leq C |\lambda|^{3/2} \|V_{\lambda}^{-1}\mathcal{T}_{\lambda} f\|_{\mathcal{H}_{\mathcal{S}}} 
\leq C \|\mathcal{T}_{\lambda} f\|_{\mathcal{H}_{\mathcal{S}}} 
\leq C \|f\|_{{\bf L}^2(\mathcal{F})}.
\end{equation}
From \eqref{EstV-2} with $(\theta,\beta)=(1/2,0)$ and \eqref{np0056},
\begin{equation}\label{21}
|\lambda|^{1/2} \|A_1^{1/2}V_{\lambda}^{-1} \mathcal{T}_{\lambda} f\|_{\mathcal{H}_{\mathcal{S}}} \leq C \|f\|_{{\bf L}^2(\mathcal{F})}.
\end{equation}
From \eqref{EstV-2} with $(\theta,\beta)=(0,0)$ and \eqref{np0056},
\begin{equation}\label{31}
|\lambda|^{1/2} \|\lambda V_{\lambda}^{-1} \mathcal{T}_{\lambda} f\|_{\mathcal{H}_{\mathcal{S}}} \leq C \|f\|_{{\bf L}^2(\mathcal{F})}.
\end{equation}
From \eqref{eq23-06-2017-1-0} and \eqref{EstV-2} with $(\theta,\beta)=(0,1/2)$
\begin{equation}\label{12}
|\lambda|^{1/2} \|W_{\lambda}V_{\lambda}^{-1}A_1 g\|_{{\bf L}^2(\mathcal{F})} 
\leq C |\lambda|^{1/2} \|V_{\lambda}^{-1}A_1 g\|_{\mathcal{H}_{\mathcal{S}}} 
\leq C \|A_1^{1/2} g\|_{\mathcal{H}_{\mathcal{S}}}.
\end{equation}
From \eqref{EstV-2} with $(\theta,\beta)=(1/2,1/2)$
\begin{equation}\label{22}
|\lambda|^{1/2} \left\|A_1^{1/2}\frac{I-V_{\lambda}^{-1}A_1}{\lambda} g \right\|_{\mathcal{H}_{\mathcal{S}}} 
\leq C \|A_1^{1/2} g\|_{\mathcal{H}_{\mathcal{S}}}.
\end{equation}
From \eqref{EstV-2} with $(\theta,\beta)=(0,1/2)$
\begin{equation}\label{32}
|\lambda|^{1/2} \left\|V_{\lambda}^{-1}A_1  g \right\|_{\mathcal{H}_{\mathcal{S}}} 
\leq C \|A_1^{1/2} g\|_{\mathcal{H}_{\mathcal{S}}}.
\end{equation}

From \eqref{eq23-06-2017-1-0} and \eqref{EstV-2} with $(\theta,\beta)=(0,0)$
\begin{equation}\label{13}
|\lambda|^{1/2} \|\lambda W_{\lambda} V_{\lambda}^{-1} h\|_{{\bf L}^2(\mathcal{F})} 
\leq C |\lambda|^{3/2} \|V_{\lambda}^{-1}h\|_{\mathcal{H}_{\mathcal{S}}} 
\leq C \|h\|_{\mathcal{H}_{\mathcal{S}}}.
\end{equation}
From \eqref{EstV-2} with $(\theta,\beta)=(1/2,0)$
\begin{equation}\label{23}
|\lambda|^{1/2} \left\|A_1^{1/2} V_{\lambda}^{-1} h\right\|_{\mathcal{H}_{\mathcal{S}}} 
\leq C \|h\|_{\mathcal{H}_{\mathcal{S}}}.
\end{equation}
From \eqref{EstV-2} with $(\theta,\beta)=(0,0)$
\begin{equation}\label{33}
|\lambda|^{1/2} \left\|\lambda V_{\lambda}^{-1} h\right\|_{\mathcal{H}_{\mathcal{S}}} 
\leq C \| h\|_{\mathcal{H}_{\mathcal{S}}}.
\end{equation}
Using  \eqref{EqresolventA0}, we deduce \eqref{ResEstA0-1}. 

\medskip

Next, let us prove \eqref{ResEstA0-3}. From \cref{prop23novBis}, \eqref{np0056}, \eqref{eq23-06-2017-1} with $\theta=1$, \eqref{EstV-2} with $(\theta,\beta)=(3/8,-1/8)$ and $(\theta,\beta)=(-1/8,-1/8)$, \eqref{EstV-2} with $(\theta,\beta)=(7/8,-1/8)$, \eqref{EstV-2} with $(\theta,\beta)=(3/8,-1/8)$, we deduce that
$$
\|(\lambda-\mathbb{A})^{-1}\mathbb{P} f\|_{{\bf H}^2(\mathcal{F})}
+\|\lambda  W_{\lambda} V_{\lambda}^{-1}\mathcal{T}_{\lambda} f\|_{{\bf H}^2(\mathcal{F})}+\|V_{\lambda}^{-1}\mathcal{T}_{\lambda} f\|_{\mathcal{D}(A_1^{7/8})}
+\|\lambda V_{\lambda}^{-1}\mathcal{T}_{\lambda} f\|_{\mathcal{D}(A_1^{3/8})} \leq C \|f\|_{{\bf L}^2(\mathcal{F})}.
$$
From \eqref{eq23-06-2017-1} with $\theta=1$, \eqref{EstV-2} with $(\theta,\beta)=(3/8,3/8)$ and $(\theta,\beta)=(-1/8,3/8)$, \eqref{EstV-2} with $(\theta,\beta)=(3/8,3/8)$,  we deduce,
\begin{equation}\label{eq25092019}
\|W_{\lambda}V_{\lambda}^{-1}A_1 g\|_{{\bf H}^2(\mathcal{F})}+\|V_{\lambda}^{-1}A_1 g\|_{\mathcal{D}(A_1^{3/8})}\leq C \|g\|_{\mathcal{D}(A_1^{5/8})}.
\end{equation}
From the relation $A_0(\lambda-A_0)^ {-1}=-I+\lambda (\lambda-A_0)^ {-1}$ with \eqref{ResEstA0-1} we first first obtain
$$
\sup_{\lambda\in \mathbb{C}_{\alpha}^+}|\lambda|^{-1/2}\|A_0(\lambda-A_0)^ {-1}\|_{\mathcal{L}(\mathcal{H})}<+\infty,
$$
and by interpolation with \eqref{ResEstA0-1} we get 
\begin{equation}\label{EstResA0-1demi}
\sup_{\lambda\in \mathbb{C}_{\alpha}^+}\|(-A_0)^{1/2}(\lambda-A_0)^ {-1}\|_{\mathcal{L}(\mathcal{H})}<+\infty.
\end{equation}
Thus using \eqref{EstResA0-1demi} we obtain  
$$
 \left\| (-A_0)^{3/4} (\lambda-A_0)^{-1}\begin{bmatrix}
0 \\ g \\ 0
\end{bmatrix}
\right\|_{\mathcal{H}}= \left\| (-A_0)^{1/2} (\lambda-A_0)^{-1}(-A_0)^{1/4}\begin{bmatrix}
0 \\ g \\ 0
\end{bmatrix}
\right\|_{\mathcal{H}}\leq C\left\| (-A_0)^{1/4}\begin{bmatrix}
0 \\ g \\ 0
\end{bmatrix}
\right\|_{\mathcal{H}}
$$
and thus from \cref{PropA} and \eqref{EqresolventA0}
$$
\|W_{\lambda}V_{\lambda}^{-1}A_1 g\|_{{\bf H}^{3/2}(\mathcal{F})}+\left\|\frac{I-V_{\lambda}^{-1}A_1}{\lambda}g\right \|_{\mathcal{D}(A_1^{7/8})}+\|V_{\lambda}^{-1}A_1 g\|_{\mathcal{D}(A_1^{3/8})}\leq C \|g\|_{\mathcal{D}(A_1^{5/8})}
$$
which gives 
$$
\left\|\frac{I-V_{\lambda}^{-1}A_1}{\lambda}g\right \|_{\mathcal{D}(A_1^{7/8})}\leq C \|g\|_{\mathcal{D}(A_1^{5/8})}.
$$
Then, since we also have \eqref{eq25092019}, we have proved
$$
\|W_{\lambda}V_{\lambda}^{-1}A_1 g\|_{{\bf H}^2(\mathcal{F})}+\left\|\frac{I-V_{\lambda}^{-1}A_1}{\lambda}g\right \|_{\mathcal{D}(A_1^{7/8})}+\|V_{\lambda}^{-1}A_1 g\|_{\mathcal{D}(A_1^{3/8})}\leq C \|g\|_{\mathcal{D}(A_1^{5/8})}.
$$
From \eqref{eq23-06-2017-1} with $\theta=1$, \eqref{EstV-2} with $(\theta,\beta)=(3/8,-1/8)$ and $(\theta,\beta)=(-1/8,-1/8)$, \eqref{EstV-2} with $(\theta,\beta)=(7/8,-1/8)$, \eqref{EstV-2} with $(\theta,\beta)=(3/8,-1/8)$, we deduce,
$$
\|\lambda W_{\lambda}V_{\lambda}^{-1}h\|_{{\bf H}^2(\mathcal{F})}+\left\|V_{\lambda}^{-1}h\right \|_{\mathcal{D}(A_1^{7/8})}+\|\lambda V_{\lambda}^{-1}h\|_{\mathcal{D}(A_1^{3/8})}\leq C \|h\|_{\mathcal{D}(A_1^{1/8})}.
$$
Then combining the above estimates we have proved 
\begin{multline}\label{ResEstA0-3-aux}
\left\| (\lambda I-A_0)^{-1}
z
\right\|_{{\bf H}^2(\mathcal{F})\times \mathcal{D}(A_1^{7/8})\times \mathcal{D}(A_1^{3/8})}
\leq C\left\|
z\right\|_{{\bf L}^2(\mathcal{F})\times \mathcal{D}(A_1^{5/8})\times \mathcal{D}(A_1^{1/8})}
\\
\left(
z
\in \mathcal{H}\cap \left({\bf L}^2(\mathcal{F})\times \mathcal{D}(A_1^{5/8})\times \mathcal{D}(A_1^{1/8})\right)\right).
\end{multline}

It remains to estimate in \eqref{ResEstA0-3} the term 
$
|\lambda|\left\| (\lambda-A_0)^{-1}z
\right\|_{{\bf L}^2(\mathcal{F})\times \mathcal{D}(A_1^{3/8})\times \mathcal{D}(A_1^{1/8})'}.
$
Assume
$$
[w,\xi_1,\xi_2]\in \mathcal{H}\cap \left({\bf H}^2(\mathcal{F})\times \mathcal{D}(A_1^{7/8})\times \mathcal{D}(A_1^{3/8})\right).
$$ 
First, from the continuity of $\Lambda^*: {\bf H}^{1/2}(\partial \mathcal{F})\to \mathcal{D}(A_1^{1/8})$ and a trace inequality we have,
$$
\|\Lambda^*(2 D(w)n)\|_{\mathcal{D}(A_1^{1/8})'}\leq C \|\Lambda^*(2 D(w)n)\|_{\mathcal{D}(A_1^{1/8})}\leq C \|D(w)n\|_{{\bf H}^{1/2}(\partial \mathcal{F})}\leq C \|w\|_{{\bf H}^{2}(\mathcal{F})}.
$$
From \eqref{RegTildeP-0} and the above estimate we deduce, 
\begin{multline*}
 \left\| A_0 \begin{bmatrix}
w \\ \xi_1 \\ \xi_2
\end{bmatrix}
\right\|_{{\bf L}^2(\mathcal{F})\times \mathcal{D}(A_1^{3/8})\times \mathcal{D}(A_1^{1/8})'} = \left\| P_0\begin{bmatrix}
 \Delta w \medskip\\ \medskip
\xi_2 \displaystyle \\ \medskip
 \displaystyle -A_1\xi_{1}-\Lambda^*(2 D(w)n)
\end{bmatrix}\right\|_{{\bf L}^2(\mathcal{F})\times \mathcal{D}(A_1^{3/8})\times \mathcal{D}(A_1^{1/8})'} \\
\leq C\left\|\begin{bmatrix}
 \Delta w \medskip\\ \medskip
\xi_2 \displaystyle \\ \medskip
 \displaystyle -A_1\xi_{1}-\Lambda^*(2 D(w)n)
\end{bmatrix}\right\|_{{\bf L}^2(\mathcal{F})\times \mathcal{D}(A_1^{3/8})\times \mathcal{D}(A_1^{1/8})'}
\leq C\left\|\begin{bmatrix}
w \\ \xi_1 \\ \xi_2
\end{bmatrix}
\right\|_{\mathcal{H}\cap \left({\bf H}^2(\mathcal{F})\times \mathcal{D}(A_1^{7/8})\times \mathcal{D}(A_1^{3/8})\right)}.
\end{multline*}
Then with formula $\lambda(\lambda-A_0)^{-1}=A_0(\lambda-A_0)^{-1}+I$ and \eqref{ResEstA0-3-aux} we deduce
$$
 |\lambda|\left\| (\lambda-A_0)^{-1}
z
\right\|_{{\bf L}^2(\mathcal{F})\times \mathcal{D}(A_1^{3/8})\times \mathcal{D}(A_1^{1/8})'}
\leq C\left\|
z
\right\|_{{\bf L}^2(\mathcal{F})\times \mathcal{D}(A_1^{5/8})\times \mathcal{D}(A_1^{1/8})},
$$
which gives the result.
\end{proof}

\section{Proof of the local in time existence for (\ref{tak2.3})}\label{sec_fix}
We prove here \cref{T01} by a fixed point argument. The proof is quite similar to the same proof for the ``flat'' case considered in \cite{plat}. For sake of completeness, we give here the main ideas of the proof.

For $R>0,T>0$ we consider the set
\begin{multline*}
\mathfrak{B}_{R,T}\ov \left\{
(F,G) \in L^2(0,T;{\bf L}^2(\mathcal{F}))\times L^2(0,T;H_{0}^{1/2}(\mathcal{I})) \ ; \  \|(F,G)\|_{L^2(0,T;{\bf L}^2(\mathcal{F}))\times L^2(0,T;H^{1/2}(\mathcal{I})))}\leq R
\right\}.
\end{multline*}

For any $(F,G)\in \mathfrak{B}_{R,T}$, we consider the solution $(w,\eta,q)$ of system \eqref{gev0.3-ter} given by \cref{T04}. In particular
\begin{equation}\label{bea0.8A}
w\in L^2(0,T;{\bf H}^2(\mathcal{F}))\cap C^0([0,T];{\bf H}^1(\mathcal{F})) \cap H^1(0,T;{\bf L}^2(\mathcal{F})),
\quad
q\in L^2(0,T;H^1(\mathcal{F}) )
\end{equation}
\begin{equation}\label{bea0.9A}
\eta \in L^2(0,T;H_{0}^{7/2}(\mathcal{I}))\cap C^0([0,T];H_{0}^{5/2}(\mathcal{I})) \cap H^1(0,T;H_{0}^{3/2}(\mathcal{I})),
\end{equation}
\begin{equation}\label{bea1.0A}
\partial_t \eta \in L^2(0,T;H_{0}^{3/2}(\mathcal{I}))\cap C^0([0,T];H_{0}^{1/2}(\mathcal{I})) \cap H^1(0,T;H_{0}^{1/2}(\mathcal{I})'),
\end{equation}
with
\begin{multline}\label{bea4.4}
\|w\|_{L^2(0,T;{\bf H}^2(\mathcal{F}))\cap C^0([0,T];{\bf H}^1(\mathcal{F})) \cap H^1(0,T;{\bf L}^2(\mathcal{F}))}
+
\|q\|_{L^2(0,T;H^1(\mathcal{F}) )},
\\
+
\|\eta\|_{L^2(0,T;H^{7/2}(\mathcal{I}))\cap C^0([0,T];H^{5/2}(\mathcal{I}))}
+
\|\partial_t \eta\|_{L^2(0,T;H^{3/2}(\mathcal{I}))\cap C^0([0,T];H^{1/2}(\mathcal{I}))}\\
\leq C_0\left(R+\|[w^0, \eta^0_1, \eta^0_2]\|_{\mathbf{H}^1(\mathcal{F})\times H^{3+\varepsilon}(\mathcal{I})\times  H^{1+\varepsilon}(\mathcal{I})}\right).
\end{multline}
Above and below, $C_0$ denotes a positive constant that depends on $\|\eta_1^0\|_{W^{7,\infty}(\mathcal{I})}$ only.

In what follows, we take $R$ (large enough) such that
\begin{equation}\label{RgeqCondIni}
R\geq 1+\|[w^0, \eta^0_1, \eta^0_2]\|_{\mathbf{H}^1(\mathcal{F})\times H^{3+\varepsilon}(\mathcal{I})\times  H^{1+\varepsilon}(\mathcal{I})}.
\end{equation}
First we notice that by interpolation, \eqref{bea4.4} yields
\begin{equation}\label{EqMehdi04012017-0}
\|\eta\|_{H^{3/4}(0,T;H^{2}(\mathcal{I}))}+\|\eta\|_{L^4(0,T;H^{3}(\mathcal{I}))}
+\|\partial_t \eta\|_{L^4(0,T;H^{1}(\mathcal{I}))}
+\|w\|_{L^{8}(0,T;{\bf H}^{5/4}(\mathcal{F}))}
\leq C_0R.
\end{equation}
The difference with respect to the proof in \cite{plat} is that here our formula for $X$ and $Y$ (see \eqref{bea3.4} and \eqref{bea3.5}) involves $\eta_1^0$. Nevertheless, one can write
$$
X(t,y_1,y_2)=\left(y_1,x_2(1+\zeta(t,x_1))\right)\quad \mbox{ and }\quad Y(t,x_1,x_2)=\left(x_1,\frac{y_2}{1+\zeta(t,y_1)}\right)
$$
where 
$$
\zeta(t,y_1)\ov \frac{\eta(t,y_1)-\eta_1^0(y_1)}{1+\eta_1^0(y_1)}.
$$
Using that $\eta_1^0\in H^{3+\varepsilon}(\mathcal{I})$, $\eta_1^0>-1$ we deduce that
$$
\frac{1}{1+\eta_1^0}\in H^{3+\varepsilon}(\mathcal{I}).
$$
Combining this with Sobolev embeddings, we deduce that 
\begin{multline}\label{arnak}
\|\zeta\|_{C^0([0,T];H^{5/2}(\mathcal{I}))}
+
\|\partial_t \zeta\|_{L^2(0,T;H^{3/2}(\mathcal{I}))\cap C^0([0,T];H^{1/2}(\mathcal{I}))}
\\
+
\|\zeta\|_{H^{3/4}(0,T;H^{2}(\mathcal{I}))}
+\|\zeta\|_{L^4(0,T;H^{3}(\mathcal{I}))}
+\|\partial_t \zeta\|_{L^4(0,T;H^{1}(\mathcal{I}))}
\leq C_0R.
\end{multline}
In particular, we have the same estimates for $\zeta$ than for $\eta$ except the estimate in $L^2(0,T;H^{7/2}(\mathcal{I}))$. In the proof of \cite{plat}, we only need the norms in \eqref{arnak} for the estimates associated with the 
change of variables (see \eqref{bea4.5}--\eqref{bea5.1}) and this allows us to prove
that for $T$ small enough, we can construct the change of variables defined in \cref{sec_construc} and consider
the mapping
\begin{equation}\label{bea0.4}
\mathcal{Z} : (F,G) \mapsto (\widehat{F}(\eta,w,q),\widehat{G}_{\eta_1^0}(\eta,w))
\end{equation}
where the maps $\widehat{F}$ and $\widehat{G}_{\eta_1^0}$ are defined by \eqref{bea4.2} and \eqref{bea4.3}, and $(w,\eta,p)$ is solution of system \eqref{gev0.3}-\eqref{bis}.
We can also show that 
\begin{equation}\label{11:15}
\|\mathcal{Z}(F,G)\|_{L^2(0,T;{\bf L}^2(\mathcal{F}))\times L^2(0,T;H^{1/2}(\mathcal{I}))}\leq C_1T^{1/8}R^{N_1},
\end{equation}
for some $N_1\geq 2$. 
More precisely, the main difference from \cite{plat} is the following:
using Proposition A.1 in \cite{plat}, \eqref{arnak} and $\zeta(0,\cdot)=0$, we deduce 
that
$$
\|\zeta\|_{L^\infty(0,T;H^2(\mathcal{I}))} \leq CT^{1/6}\|\zeta\|_{H^{3/4}(0,T;H^{2}(\mathcal{I}))} \leq CT^{1/6}R.
$$
This yields 
$$
\|\nabla Y (X)-I_2\|_{L^\infty(0,T;L^\infty(\mathcal{F})^4)}
+\|\det(\nabla X)-1\|_{L^\infty(0,T;L^\infty(\mathcal{F}))}
\leq CT^{1/6}R
$$
instead of (6.16) in \cite{plat}.
Then, following the computation in \cite{plat} we deduce \eqref{11:15}.

From \eqref{11:15}, for all $T\leq C_1^{-8}R^{8-8N_1}$ we have 
$$
\mathcal{Z}(F,G)\in \mathfrak{B}_{R,T}.
$$

Similarly, taking $T$ possibly smaller, we can also show that 
$\mathcal{Z}$ is a strict contraction on $\mathfrak{B}_{R,T}$ and using the Banach fixed point theorem, we deduce the existence and uniqueness of $(F,G)\in \mathfrak{B}_{R,T}$
such that
$$
\mathcal{Z}((F,G))=(F,G).
$$
The corresponding solution $(\eta,w,q)$ of system \eqref{gev0.3}-\eqref{bis} is a solution of \eqref{tak2.3-CV2}--\eqref{su4.3-CV-old}.

The proof of the uniqueness is similar to the proof of uniqueness given in \cite{plat}.

\appendix

\section{Formula for the change of variables}\label{sec_formulas}
Let us give some formulas for the change of variables 
$$
X(t,y_1,y_2)=\left(y_1,x_2(1+\zeta(t,x_1))\right)\quad \mbox{ and }\quad Y(t,x_1,x_2)=\left(x_1,\frac{y_2}{1+\zeta(t,y_1)}\right)
$$
that are used in \cref{sec_fix} for the fixed point.
We also need these formulas for the study of the linear system (see \cref{toflat}), and in that case, $\zeta=\eta_1^0$, $X=\widetilde X$, $a=\Tilde a$ and $b=\Tilde b$ are independent of time.
\begin{equation}\label{bea4.5}
\nabla X (t,y_1,y_2)= 
\begin{bmatrix}
1 & 0 \\
y_2 \partial_s \zeta & 1+\zeta
\end{bmatrix},
\quad 
b(t,y_1,y_2)= 
\begin{bmatrix}
1+\zeta & 0 \\
-y_2 \partial_s \zeta & 1
\end{bmatrix},
\end{equation}
\begin{equation}\label{bea4.6}
\nabla Y (t,x_1,x_2)= 
\begin{bmatrix}
1 & 0 \\
-x_2 \frac{\partial_s \zeta}{(1+\zeta)^2} & \frac{1}{1+\zeta}
\end{bmatrix},
\quad 
a(t,x_1,x_2)= 
\begin{bmatrix}
 \frac{1}{1+\zeta} & 0 \\
x_2 \frac{\partial_s \zeta}{(1+\zeta)^2} &1
\end{bmatrix}.
\end{equation}
\begin{equation}\label{bea4.7}
a(X)= 
\begin{bmatrix}
 \frac{1}{1+\zeta} & 0 \\
y_2 \frac{\partial_s \zeta}{1+\zeta} &1
\end{bmatrix},
\quad
\nabla Y (X)= 
\begin{bmatrix}
1 & 0 \\
-y_2 \frac{\partial_s \zeta}{1+\zeta} & \frac{1}{1+\zeta}
\end{bmatrix},
\end{equation}
\begin{equation}\label{bea5.0}
\nabla Y (X)-I_2= 
\begin{bmatrix}
0 & 0 \\
-y_2 \frac{\partial_s \zeta}{1+\zeta} & \frac{-\zeta}{1+\zeta}
\end{bmatrix},
\quad
\det(\nabla X)=1+\zeta,
\end{equation}
\begin{equation}\label{bea4.8}
\frac{\partial a}{\partial x_1}(X)= 
\begin{bmatrix}
 \frac{-\partial_s \zeta}{(1+\zeta)^2} & 0 \\
y_2 \frac{\partial_{ss} \zeta(1+\zeta)-2(\partial_{s} \zeta)^2 }{(1+\zeta)^2} &0
\end{bmatrix},
\quad
\frac{\partial a}{\partial x_2}(X)= 
\begin{bmatrix}
0 & 0 \\
\frac{\partial_s \zeta}{(1+\zeta)^2} &0
\end{bmatrix},
\end{equation}
\begin{equation}\label{bea4.8CC}
\frac{\partial^2 a}{\partial x_1\partial x_2}(X)= 
\begin{bmatrix}
 0 & 0 \\
\frac{\partial_{ss} \zeta(1+\zeta)-2(\partial_{s} \zeta)^2 }{(1+\zeta)^3} &0
\end{bmatrix},
\end{equation}
\begin{equation}\label{bea4.9}
\frac{\partial^2 a}{\partial x_1^2}(X)= 
\begin{bmatrix}
 \frac{-\partial_s \zeta}{(1+\zeta)^2} & 0 \\
y_2 \frac{\partial_{sss} \zeta(1+\zeta)^2-6(1+\zeta)\partial_{s} \zeta\partial_{ss}\zeta+6(\partial_s\zeta)^3 }{(1+\zeta)^3} &0
\end{bmatrix},
\quad
\frac{\partial^2 a}{\partial x_2^2}(X)= 0,
\end{equation}
\begin{equation}\label{bea5.1}
\frac{\partial}{\partial x_1}\nabla Y (X)= 
\begin{bmatrix}
0 & 0 \\
y_2 \frac{-\partial_{ss} \zeta(1+\zeta)+2(\partial_s \zeta)^2}{(1+\zeta)^2} & \frac{-\partial_s \zeta}{(1+\zeta)^2}
\end{bmatrix},
\quad
\frac{\partial}{\partial x_2}\nabla Y (X)= 
\begin{bmatrix}
0 & 0 \\
- \frac{\partial_s \zeta}{(1+\zeta)^2} & 0
\end{bmatrix}.
\end{equation}
\begin{equation}\label{bea5.2}
\partial_t a(X)= 
\begin{bmatrix}
 \frac{-\partial_t \zeta}{(1+\zeta)^2} & 0 \\
y_2 \frac{\partial_{ts} \zeta(1+\zeta)-2\partial_s \zeta\partial_t \zeta}{(1+\zeta)^2} & 0
\end{bmatrix},
\quad
\partial_t Y(X)
=\begin{bmatrix}
0
\\
 -y_2\frac{\partial_t \zeta}{1+\zeta}
\end{bmatrix}.
\end{equation}

We also recall here how to obtain formulas \eqref{formb}: differentiating \eqref{gev8.1}, we deduce successively 
$$
\frac{\partial w_i}{\partial x_j}=\sum_{k} \frac{\partial \widetilde a_{ik}}{\partial x_j} \widetilde{w}_k(\widetilde Y)
+
\sum_{k,\ell} \widetilde a_{ik} \frac{\partial \widetilde{w}_k}{\partial y_\ell}(\widetilde Y) \frac{\partial \widetilde Y_\ell}{\partial x_j}
$$
and
\begin{multline*}
\frac{\partial^2 w_i}{\partial x_j^2}
= \sum_{k} \frac{\partial^2 \widetilde a_{ik}}{\partial x_j^2} \widetilde{w}_k(\widetilde Y)
+2\sum_{k,\ell} \frac{\partial \widetilde a_{ik}}{\partial x_j} \frac{\partial \widetilde{w}_k}{\partial y_\ell}(\widetilde Y) \frac{\partial \widetilde Y_\ell}{\partial x_j}
\\
+\sum_{k,\ell,m}  \widetilde a_{ik}\frac{\partial^2 \widetilde{w}_k}{\partial y_\ell\partial y_m}(\widetilde Y) 
	\frac{\partial \widetilde Y_\ell}{\partial x_j} \frac{\partial \widetilde Y_m}{\partial x_j}
+\sum_{k,\ell} \widetilde a_{ik} \frac{\partial \widetilde{w}_k}{\partial y_\ell}(\widetilde Y) \frac{\partial^2 \widetilde Y_\ell}{\partial x_j^2}.
\end{multline*}
Composing by $\widetilde Y$ and multiplying by $\widetilde b_{\alpha i}$, we deduce the first formula of \eqref{formb}. The second one can be done in a similar way.

\bibliography{references}
\bibliographystyle{plain}
\end{document}